\documentclass{article}
\usepackage{bm}
\usepackage{amsmath}
\usepackage{amsthm}
\usepackage{amssymb}
\usepackage{appendix}
\usepackage{graphicx} 
\usepackage{subcaption}
\usepackage{comment}
\usepackage[colorlinks=true, linkcolor=blue, citecolor=blue, urlcolor=blue]{hyperref}
\usepackage{float}
\usepackage{placeins}
\usepackage{amsfonts}
\usepackage{xcolor}
\usepackage{csquotes}
\usepackage{bm}
\newtheorem{theorem}{Theorem}
\numberwithin{theorem}{section}
\newtheorem{lemma}{Lemma}
\numberwithin{equation}{section}
\usepackage{algorithm}
\usepackage{algorithmic}
\usepackage{xurl}

\DeclareMathOperator*{\argmin}{arg\,min} 
\DeclareMathOperator*{\argmax}{arg\,max}
\usepackage{chngcntr} 
\counterwithin{figure}{section}

\counterwithin{table}{section}

\usepackage{listings}
\usepackage{xcolor} 
\newcommand{\sg}[1]{{\color{red}#1}}
\newcommand{\shuai}[1]{{\color{cyan}#1}}

\title{Uniform-in-time rational approximation of the matrix exponential with real poles}
\author{Stefan G\"uttel\thanks{Department of Mathematics, The University of Manchester, Oxford Road, Manchester, M13\,9PL, United Kingdom, \texttt{stefan.guettel@manchester.ac.uk}} \and Shuai Shao\thanks{Department of Mathematics, The University of Manchester, Oxford Road, Manchester, M13\,9PL, United Kingdom, \texttt{shuai.shao-2@manchester.ac.uk}}}

\begin{document}

\maketitle

\begin{abstract}
  We propose two new approaches for constructing families of rational functions with shared real poles that nearly uniformly approximate the functions $\exp(-tz)$ for  $z\geq 0$ and $t$ in a positive time interval. The first result concerns the case where all real poles coalesce into a single point. With an appropriate choice of a weight function we are able to derive a closed formula for the asymptotically optimal location of such a pole. We then discuss the more general case where all real poles are distinct. Using Zolotarev's construction of certain optimal rational functions, we present a simple algorithm to derive nearly optimal poles efficiently. 
  We analyze the stability of the numerical evaluation of the resulting rational matrix functions in floating-point arithmetic. By controlling the growth of potential ill-conditioning arising from partial fractions, reliable and highly parallelizable exponential propagators are obtained.
\end{abstract}

\section{Introduction}
The efficient numerical approximation of $\exp(-tA)\bm{b}$, where $A \in \mathbb{R}^{N\times N}$ is a symmetric positive semidefinite matrix and $\bm{b} \in \mathbb{R}^N$ is a vector, for all $t$ in a positive time interval $T = [t_\mathrm{min},t_\mathrm{max}]$, is a fundamental problem in science and engineering. Such expressions arise as solutions of linear evolution problems $\bm{u}'(t) + A\bm{u}(t) = 0$ with initial condition $\bm{u}(0) = \bm{b},$ and appear in applications ranging from geophysical modeling~\cite{{borner2015three},{qiu2019block},{borner2025fast}} and chemical kinetics~\cite{{munsky2006finite},{jahnke2007solving}} to network analysis~\cite{{estrada2012physics},{benzi2013ranking}} and exponential integrators~\cite{saad1992analysis,hochbruck2010exponential,guttel2010rational,guttel2013rational}. 

Two primary challenges arise in addressing this problem. First, explicitly computing and storing the matrix exponential $\exp(-tA)$ is computationally prohibitive if $N$ is sufficiently large. Second, constructing a separate approximation for each time point can be costly, especially when solutions are required at a large number of time points in~$T$. 

An effective way to overcome these challenges is to construct a family of \textit{shared-pole} (or \emph{shared-denominator}) rational approximants 
\begin{equation}\label{partial_fraction_form}
    r_{t,n}(z) = \sum_{i=1}^n\frac{\alpha_i(t)}{z-\sigma_i}
\end{equation}
such that 
\begin{equation}\label{approx}
    r_{t,n}(A)\bm{b} \approx \exp(-tA)\bm{b}
\end{equation}
uniformly for all $t \in T$. We refer to $n$ as the degree of approximation. A crucial point is that the poles $\sigma_i$ of $r_{t,n}$ are independent of the time parameter~$t$. Therefore, only $n$ shifted linear systems $(A-\sigma_i I)\bm{x}_i = \bm{b}$ need to be solved in total, regardless of the number of time points $t$ one wants to evaluate for. 

To assess the quality of approximation in~\eqref{approx}, we assume without loss of generality that $\Vert \bm{b}\Vert_2=1$, and bound 
\begin{eqnarray}
   & & \max_{t\in T}\left\| r_{t,n}(A)\bm{b}-\exp(-tA)\bm{b}\right\|_2 \nonumber \\ 
   &\leq& \max_{t\in T}\max_{z \geq 0}\left| r_{t,n}(z)-\exp(-tz)\right| =: E_{n, T}. \label{time_uni_err}
\end{eqnarray}
We refer to $E_{n, T}$ as the \emph{time-uniform error} for the degree $n$ rational approximation in the time domain $T.$ This bound uniformly holds for \emph{any} symmetric positive semidefinite matrix $A$, which leads to mesh-independent exponential integrators when $A$ arises from discretizations of differential operators.

Reliable algorithms for computing shared complex-pole rational approximation are readily available, such as vector fitting~\cite{gustavsen2002rational}, rational Krylov fitting (\mbox{RKFIT})~\cite{berljafa2017rkfit}, the set-valued adaptive Antoulas--Anderson method~\cite{lietaert2022automatic}, and shared-pole Carath{\'e}o\-dory--Fej{\'e}r  approximation~\cite{al2025shared}. Shared-pole approximants with complex poles can also be constructed using contour integration techniques; see~\cite{{lopez2004numerical},{lopez2006spectral},{weideman2007parabolic}}. However, the accuracy of the resulting approximants generally deteriorates rather quickly as one moves away from the time point the contours are optimized for; see, e.g., Fig.~6.2 in~\cite{berljafa2017rkfit} for a direct comparison of contour-derived poles against RKFIT optimized poles. All of the mentioned algorithms can return poles in conjugate pairs in the complex plane, which can be exploited to reduce the number of required linear system solves from $n$ to $\lceil \frac{n}{2} \rceil$. However, the computational cost of solving a complex shifted linear system is typically much higher than that of solving a real-valued linear system. In particular, when all the poles~$\sigma_i$ are negative, the associated linear systems are positive definite. In that case, either the direct solution via Cholesky factorization or an iterative method like preconditioned CG can be used. Therefore, it is highly beneficial to compute rational approximants with shared \emph{real} poles. This problem has received considerably less attention in the literature. Exceptions include the work of Druskin et al.~\cite{druskin2009solution} and our extension of Andersson's approach~\cite{guttel2025concentrated}. From an algorithmic perspective, dictionary-based methods may also be applicable, including approaches based on the rational empirical interpolation method (rEIM)~\cite{li2025new}.



In the recent work~\cite{guttel2025concentrated} we studied  the case where all shared real poles are concentrated; that is, $\sigma_k =\sigma < 0$ for all  $k=1,\ldots, n$. Extending a classical result by Andersson~\cite{andersson1981approximation}, we proposed an optimal location of $\sigma$ that asymptotically minimizes the time-uniform error. The fact that all poles are concentrated at $\sigma$ means that, with direct solvers, only one Cholesky factorization of $A - \sigma I$ needs to be computed and can be reused for all $n$ solves. The determination of $\sigma$ in~\cite{guttel2025concentrated} required the numerical minimization of a two-parameter scalar function. 

As a first main contribution in this work, described in section~\ref{sec:weight}, we will reconsider the concentrated-pole approximation problem but with a weight function. That is, we seek to minimize a weighted version of \eqref{time_uni_err},
\begin{eqnarray}
   \max_{t\in T} \left( w(t)^n \cdot \max_{z \geq 0}\left| r_{t,n}(z)-\exp(-tz)\right| \right)=: E^w_{n, T}. \label{time_uni_err_weighted}
\end{eqnarray}
We show that for any weight function $w(t) = t^\gamma$ with $\gamma \geq 0.1684$, the asymptotically optimal concentrated pole is given by
\[
\sigma^* = \frac{-n}{\sqrt{2}\,t_{\max}}.
\]
That is, if the weight function increases sufficiently fast over $[t_\mathrm{min},t_\mathrm{max}]$, it is fine to optimize only for the right-most time point~$t_\mathrm{max}$. Furthermore, this explicit formula for $\sigma^*$ completely removes the need for numerical optimization (improving on  \cite{guttel2025concentrated}). Weight functions of the form $w(t) = t^\gamma$ arise very naturally in applications such as geophysical electromagnetic modeling~\cite{borner2025fast}, where simulated electromagnetic fields are known to decay geometrically and the relative error of the transient needs to be uniformly small over the simulation time interval.

Next, in section~\ref{sec:zolo}, we turn our attention to the case where the poles of the functions $r_{t,n}$ are real but distinct. Our approach is based on new family of real-pole rational interpolants derived from  Zolotarev's optimal functions. While earlier work by Druskin et al.~\cite{druskin2009solution} proposed an optimal shared real pole selection also based on Zolotarev functions, their work requires $A$ to have eigenvalues in a finite positive spectral interval, with the approximation quality deteriorating with increasing condition number of $A$. By contrast, our pole selection approach works for \emph{any} positive semidefinite matrix~$A$ independent of conditioning. 

There is great benefit in evaluating the action of a rational matrix function on a vector, $r_{t,n}(A)\bm{b}$, in partial fraction form, as each linear solve can be distributed to a dedicated processor. 
It is well known that partial fraction decompositions suffer numerical instability in particular when some poles of the rational function are clustered. This is analyzed  in section~\ref{error_anal}. We show  that the  instability can be controlled, allowing for a practical and reliable numerical integrator with high-level of parallelism. 

In section~\ref{exper} we present numerical experiments  to illustrate our analysis and compare real-pole approximants to other choices of rational families computed by \mbox{RKFIT}~\cite{berljafa2017rkfit}. We conclude in  section~\ref{concl}.

\section{Concentrated real pole approximation with a weight function}\label{sec:weight}
For any $q > 0$, let $\mathcal{R}^n_q$ denote the set of degree $n$ real rational functions with all poles concentrated at $-nq $,  
\[
    \mathcal{R}^n_q = \left\{\frac{p_n(z)}{(z+nq)^n}:p_n \in \mathcal{P}_n\right\}.
\]
Here, $\mathcal{P}_n$ denotes the class of polynomials of degree at most~$n$. Moreover, for each time point $t$, let 
\[
    \rho^t_n(q) = \inf\left\{\sup_{z \geq 0}\left| r_{t,n}(z)-\exp(-tz) \right|: r_{t,n} \in \mathcal{R}^n_q\right\}.
\] 
The time-weighted uniform error~\eqref{time_uni_err_weighted} associated with concentrated poles $-nq$ can then be expressed as
\begin{eqnarray}\label{time_weighted_err_q}
   E^w_{n, T}(q) = \max_{t\in T}  \left( w(t)^n \cdot \rho^t_n(q) \right). 
\end{eqnarray}
Theorem~\ref{thm1} below provides the formula for the asymptotic minimizer of~\eqref{time_weighted_err_q}.
\begin{theorem}\label{thm1}
    Let $t \in T = [t_{\min}, t_{\max}]$ with $0 \leq t_{\min} \leq t_{\max}$ and let $w(t) = t^{\gamma}$ with $\gamma > 0$. Then, 
    \begin{eqnarray}
    \lim_{n \rightarrow +\infty}\left(\argmin_{q > 0} E^w_{n, T}(q)\right) = \argmin_{q > 0} \left(q^{-\gamma} \max_{z \in [qt_{\min}, qt_{\max}]}z^{\gamma}\widetilde{H}(z)\right), \label{formula_min}
    \end{eqnarray}
    where $z = qt$ and 
\[
   \widetilde{H}(z) = \exp \left(\log\left|\frac{\widetilde{z}-1}{\widetilde{z}+1}\right|+z \Re(\widetilde{z}^2)\right),
\]
with $\widetilde{z}$ being the root in $\Im(\zeta) \geq 0$ of the equation
\[
        z(\zeta^3-\zeta)+1 = 0
\]
that has the smallest positive real part.
\end{theorem}
\begin{proof}
    Based on the work of  Andersson~\cite{andersson1981approximation}, we have that for any $q > 0$, 
\[
    \displaystyle\lim_{n\rightarrow +\infty}\rho^1_n(q)^{1/n} = \widetilde{H}(q).
\]
   Using the argument scaling property $\rho^t_n(q) = \rho^1_n(qt)$, we have 
   \begin{align*}
       \lim_{n \rightarrow +\infty}\left(\argmin_{q > 0} E^w_{n, T}(q)\right) &=  \lim_{n \rightarrow +\infty}\left(\argmin_{q > 0} \max_{t\in T}  \left( w(t)^n \cdot \rho^t_n(q) \right)\right) \\ &= \lim_{n \rightarrow +\infty} \left(\argmin_{q > 0} \max_{t\in T}  \left( w(t) \cdot \widetilde{H}(qt)\right)^n\right) \\ &= \argmin_{q > 0} \max_{t\in T}  \left( t^{\gamma} \cdot \widetilde{H}(qt)\right) \\
        &= \argmin_{q > 0} \left(q^{-\gamma} \max_{z \in [qt_{\min}, q_{t_{\max}}]}z^{\gamma}\widetilde{H}(z)\right).
   \end{align*}
\end{proof}

Although the right-hand side of~\eqref{formula_min} is generally obtained via numerical optimization~\cite{guttel2025concentrated}, this optimization step can be bypassed when $\gamma$ exceeds a certain threshold, as shown in Theorem~\ref{0.1684}.
\begin{theorem}\label{0.1684}
For $\gamma\geq 0.1684$, the function $F(z)=z^\gamma \widetilde H(z)$ is monotonically increasing for $z > 0$. In this case, the asymptotic minimizer in Theorem~\ref{thm1} is 
\[
\lim_{n \rightarrow +\infty}\left(\argmin_{q > 0} E^w_{n, T}(q)\right) = \frac{1}{\sqrt{2}\,t_{\max}}.
\]
\end{theorem}
\begin{proof}
The derivative of $F$ is given by
\[
F'(z) = z^{\gamma-1} \left( \gamma \widetilde{H}(z) + z \widetilde{H}'(z) \right).
\]
We have shown in~\cite{guttel2025concentrated} that $\widetilde{H}(z)>0$ for $z >0$ and $\widetilde{H}'(z) \geq 0$ for $z \geq 1/\sqrt{2}$. Therefore, we have
$F'(z) > 0$ for $z \geq 1/\sqrt{2}$.

Now, let us show that $F'(z)>0$ for $0<z<1/\sqrt{2}$. Let $G(z) =\gamma \widetilde{H}(z) + z \widetilde{H}'(z)$, it suffices to show that $G(z) > 0$ for $0<z < 1/\sqrt{2}$.
Applying the results of the equations (9) and (10) in Theorem~2.3 from \cite{guttel2025concentrated}, we have
\begin{align}\label{shuai:1}
\widetilde{H}\left(z\right)=\widetilde{h}\left(m\right)=\sqrt{\frac{m+1}{m-1}}\exp\left(\frac{m^2-2}{2m\left(m^2-1\right)}\right),
\end{align}
where $z$ and $m$ are related by
\begin{align}\label{shuai:2}
zm\left(m^2-1\right)=-1.
\end{align}
Therefore,  
\begin{align}\label{shuai:3}
\widetilde{h}'\left(m\right)=\widetilde{h}\left(m\right)\left(s'\left(m\right)-\frac{1}{m^2-1}\right),
\end{align}
where
\begin{align}\label{shuai:4}
s(m)=\frac{m^2-2}{2m\left(m^2-1\right)}.
\end{align}
By taking the implicit derivative with respect to $z$ in~\eqref{shuai:2}, we have
\[
z\left(3m^2-1\right)m'(z)+m\left(m^2-1\right)=0,
\]
and hence
\begin{align}\label{shuai:6}
m'\left(z\right)=\frac{1}{z^2\left(3m^2-1\right)}.
\end{align}
From~\eqref{shuai:1}, we can evaluate
\begin{align}
\widetilde{H}'\left(z\right)=\widetilde{h}'\left(m\right)m'\left(z\right),
\end{align}
with the results in~\eqref{shuai:3} and~\eqref{shuai:6}. Substituting into $G$, we obtain
\[
G\left(z\right)=\widetilde{h}\left(m\right)\left(\gamma+z\left(s'\left(m\right)-\frac{1}{m^2-1}\right)m'\left(z\right)\right).
\]
Recall that we want to show that $G(z)>0$ for $0<z<1/\sqrt2$. Since $\widetilde{h}\left(m\right)>0$, it suffices to show that
\begin{align}\label{shuai:8}
\gamma+z\left(s'\left(m\right)-\frac{1}{m^2-1}\right)m'\left(z\right)>0.
\end{align}
By~\eqref{shuai:2}, this is equivalent to the condition  $m<-\sqrt2$. The inequality~\eqref{shuai:8} is equivalent to
\[
\gamma z+z^2m'\left(z\right)\left(s'\left(m\right)-\frac{1}{m^2-1}\right)>0,
\]
and hence
\[
\frac{-\gamma}{m\left(m^2-1\right)}+\frac{1}{3m^2-1}\left(s'\left(m\right)-\frac{1}{m^2-1}\right)>0.
\]
Substituting $s\left(m\right)$ in~\eqref{shuai:4}, we have to show that
\begin{align}\label{shuai:11}
\phi(m) >-\gamma
\end{align}
for all $m<-\sqrt2$, where
\[
\phi(m) = \frac{m\left(m^2-1\right)}{3m^2-1}\left(\frac{m^4-5m^2+2}{2m^2\left(m^2-1\right)^2}+\frac{1}{m^2-1}\right).
\]
Since 
\[
    \phi'(m)=-\frac{m^4-5m^2+2}{2m^2\left(m^2-1\right)^2},
\]
we have that $\phi(m)$ decreases for $m \in (-\infty, m^*)$ and increases for $m \in [m^*, -\sqrt{2}) $ where $$m^* = -\sqrt{\frac{5+\sqrt{17}}{2}}.$$ Hence, 
\[
    \min_{m<-\sqrt{2}} \phi(m) = \phi(m^*) \approx -0.1684. 
\]
Therefore,~\eqref{shuai:11} holds for all $\gamma \geq 0.1684$. In this case, we apply Theorem~\ref{thm1} and obtain
\begin{align*}
\lim_{n \rightarrow +\infty}\left(\argmin_{q > 0} E^w_{n, T}(q)\right) &= \argmin_{q > 0} \left(q^{-\gamma} \max_{z \in [qt_{\min}, q_{t_{\max}}]}z^{\gamma}\widetilde{H}(z)\right) \\
&= \argmin_{q > 0} \left(t_{\max}^{\gamma} \widetilde{H}(qt_{\max})\right) \\ &= \argmin_{q > 0} \widetilde{H}(qt_{\max}) \\ &= \frac{1}{\sqrt{2}\,t_{\max}}.
\end{align*}
The last equality follows from the fact that $1/\sqrt{2}$ is the unique minimizer of $\widetilde{H}$ (see~\cite{andersson1981approximation} and also~\cite{guttel2025concentrated}). This completes the proof.   

\end{proof}

To summarize, 
Theorem~\ref{0.1684} tells us that, if the weight function increases sufficiently fast over the time interval $[t_\mathrm{min},t_\mathrm{max}]$, it is fine to optimize only for the last time point~$t_\mathrm{max}$, and the asymptotically optimal shared concentrated real pole for minimizing the time-weighted uniform error~\eqref{time_uni_err_weighted} is 
\begin{equation}\label{opt_poles}
\lim_{n \rightarrow +\infty}\sigma^*(n) = \frac{-n}{\sqrt{2}\,t_{\max}}.
\end{equation}

\section{A Zolotarev-based pole selection}\label{sec:zolo}
So far, we have only discussed the case where all real poles of rational approximants are concentrated. While this is a reasonable choice for approximation on short time intervals, its approximation quality deteriorates when the time intervals become longer. This motivates the subsequent approach of using \emph{distinct real poles,} which can achieve a much higher level of accuracy for large time intervals due to the increased flexibility of poles being distributed on the negative real axis. 

We first introduce our main algorithm, which selects a suitable set of real poles that approximately minimizes the time-uniform error $E_{n, T}$ for any degree of approximation $n$ and time interval $T.$ The core ingredient is the Hermite integral formula for the error of rational interpolants, which dates back to Walsh~\cite{walsh1935interpolation}. For each $t\in T,$ let $r_{t,n}(z)$ be the degree $n$ rational interpolant of $f_t(z) := \exp(-tz)$, with interpolation points $\vartheta_i \in [0, +\infty)$ and poles $\sigma_i \in [c,d]<0$.  Then 
\begin{equation}\label{herm_err}
f_t(z) - r_{t,n}(z)
= \frac{s_n(z)}{2\pi i}
\int_{\Gamma} \frac{f_t(\zeta)}{s_n(\zeta)(\zeta - z)} \, d\zeta .
\end{equation}
Here, $\Gamma$ is any contour that encloses the interpolation points  and $s_n$ is a rational nodal function defined as
\begin{equation}\label{nodal_fun}
s_n(z)
= \frac{\prod_{i=1}^n (z - \vartheta_i)}{\prod_{i=1}^n (z - \sigma_i)}.
\end{equation}
The error formula~\eqref{herm_err} indicates that the time-uniform error $E_{n, T}$ can be effectively controlled by bounding the growth of $\vert s_n(z) \vert$ on $z\geq 0.$ This can be achieved by solving a third Zolotarev approximation problem on the condenser $\Xi_1 := [c,d] \cup [0, +\infty)$, which yields a set of interpolation points $\vartheta_i$ and poles $\sigma_i$ that minimizes the ratio
\begin{equation}\label{zolo3}
    \frac{\max_{z\in [0, +\infty)}{\left| s_n(z)\right|}}{\min_{z\in [c,d]}{\left| s_n(z)\right|}}.
\end{equation}

Druskin, Knizhnerman and Zaslavsky~\cite{druskin2009solution} investigated a similar  problem for a symmetric condenser consisting of an interval on the positive real axis and its reflection about the origin. 
In our problem, we can map the original condenser to a symmetric one by a M\"obius transformation and obtain $\vartheta_i$ and $\sigma_i$ using the method proposed in~\cite[section~4]{druskin2009solution}. We then map $\vartheta_i$ and $\sigma_i$ back to the original condenser by an inverse M\"obius transformation. The resulting rational nodal function $s_n$ satisfies  
\begin{equation}\label{constr}
    \max_{z\in [0, +\infty)}{\left| s_n(z)\right|} = 1.
\end{equation}

Algorithm~\ref{algo} details the implementation of this idea. Lines 1 to 2 give the mapping $\psi(z)=(z-\varsigma)/(z-\varrho)$ with specifically chosen $\varsigma$ and $\varrho.$ This maps our condenser $\Xi_1$ to the symmetric condenser $\Xi_2 :=[-\eta_2, -\eta_1] \cup [\eta_1, \eta_2] $ with $\eta_1=-\psi(d)$ and $\eta_2=-\psi(c)$. Lines 3 to 7 follow the procedures proposed in~\cite{druskin2009solution}, applied to the condenser~$\Xi_2$, which yields poles and zeros on $\Xi_2$. Line~8 maps these poles and zeros back to the original condenser $\Xi_1$ via $\psi^{-1}.$


\begin{algorithm}[H]
\caption{Choosing Zolotarev poles in a fixed interval}\label{algo}
    \textbf{Input:}  Interval $[c, d]$ with $c < d <0$, degree of approximation $n\in \mathbb{N}^+$.\\
    \textbf{Output:} Poles $\{\sigma_i\}_{i=1}^n\subset [c,d]$ and zeros $\{\vartheta_i\}_{i=1}^n\subset [0, +\infty)$ of $s_n$ that minimizes~\eqref{zolo3} subject to~\eqref{constr}.
\begin{algorithmic}[1]
        \STATE Compute $\varsigma \gets c+\sqrt{c(c-d)}$ and $\varrho\gets 2c-\varsigma.$
        \STATE Compute $\eta_1\gets (\varsigma-d)/(d-\varrho)$, $\eta_2\gets (\varsigma-c)/(c-\varrho)$, and $\mu \gets 1-(\eta_1/\eta_2)^2.$ 
        \STATE Evaluate the integral $J \gets \int_{0}^{1}{\left((1-x^2)(1-\mu x^2)\right)^{-1/2}dx}.$
        \FOR{$i = 1$ to $n$}
        \STATE Compute $v\gets (2n-2i+1)J/(2n).$
        \STATE Find $\theta$ such that $\int_0^\theta \left(1-\mu\sin^2{\zeta}\right)^{-1/2}d\zeta = v.$
        \STATE Obtain $\vartheta_i \gets \eta_2(1-\mu\sin^2{\theta})^{1/2}.$
        \STATE Obtain $\sigma_i \gets (\varsigma+\varrho \vartheta_i)/(1+\vartheta_i)$ and $\vartheta_i \gets (\varsigma-\varrho \vartheta_i)/(1-\vartheta_i).$
        \ENDFOR
\end{algorithmic}
\end{algorithm}
It remains choosing an optimal pole interval $[c,d]$. According to classical results by Andersson~\cite{andersson1981approximation} and Borwein~\cite{borwein1983rational}, the asymptotically optimal real poles of the degree~$n$ rational approximation of $f_t(z) = \exp(-tz)$ on $z\geq 0$ should be concentrated at $-n/(\sqrt{2}t)$ for any $t > 0$. Therefore, a suitable choice initial guess for a pole interval is $[c_0, d_0] = [-n/(\sqrt{2}\,t_{\min}), -n/(\sqrt{2}\,t_{\max})]$ where $T = [t_{\min}, t_{\max}].$ We then refine this  pole interval $[c,d]$ by numerically solving a two-parameter optimization problem:
\begin{equation}\label{minimize}
    \mbox{Find } (c^*, d^*) = \argmin_{(c,d)\in \Omega}\max_{t\in T}\max_{z\in Z}\left| f_t(z)-r^{(c,d)}_{t,n}(z)\right|, 
\end{equation}
where $\Omega := \{(c,d)\in\mathbb{R}^2 : c<d<0\}$ is the feasible set, $Z:=[z_1, z_2, \ldots, z_m]^T$ is a vector of $m \gg n$ sufficiently dense sample points on the nonnegative real axis, and $r^{(c,d)}_{t,n}$ is a degree $n$ rational interpolant of $f_t$ with poles $\sigma_i \in [c,d]$ and interpolation points $\vartheta_i\in[0,+\infty)$ obtained from Algorithm~\ref{algo}. To enforce the interpolatory condition, we compute the vector of residues $\bm{\alpha}(t):=[\alpha_1(t),\; \cdots\; , \alpha_n(t)]^T$ of $r^{(c,d)}_{t,n}$ by solving a linear system 
\begin{equation}\label{linsys}
    C\bm{\alpha}(t) = \bm{f}(t)
\end{equation}
involving a Cauchy matrix $C\in \mathbb{R}^{n\times n}$ with general entries $C_{ij}=(\vartheta_i-\sigma_j)^{-1}$ for $1\leq i,j\leq n$ and a vector $\bm{f}(t):=[f_t(\vartheta_1),\; \cdots\;, f_t(\vartheta_n)]^T \in \mathbb{R}^n$.


Since the degree~$n$ is usually small to moderate, the linear system~\eqref{linsys} can be solved to high accuracy using multiprecision (e.g., in MATLAB using the Advanpix Multiprecision Toolbox~\cite{MCT}). Alternatively, by exploiting the structure of the Cauchy matrix, the system~\eqref{linsys} can be solved more accurately via a rank-revealing decomposition~\cite{dopico2012accurate}. Our numerical experiments indicate that this approach attains accuracy nearly indistinguishable from the multiprecision method for small to medium values of $n$, but begins to exhibit numerical instability for larger values of~$n$. Therefore, we recommend employing multiprecision arithmetic with $128$-digits for solving~\eqref{linsys}. Again, this one-off optimization does not create noticeable overhead as $n$ is small. 


To solve the optimization problem~\eqref{minimize} in MATLAB,  we use the \texttt{fminsearch} function. This algorithm is based on the Nelder–Mead simplex method~\cite{nelder1965simplex}, which determines the search directions by comparing the function values at the vertices of simplexes. This is suitable for solving problem~\eqref{minimize} where explicit analytical derivatives are not easily available. The overall method is summarized in Algorithm~\ref{algo2}.

\begin{algorithm}[H]
\caption{Choosing Zolotarev poles in a refined interval}\label{algo2}
    \textbf{Input:}  Degree of approximation $n\in \mathbb{N}^+$, a vector of $m \gg n$ real positive sample points $Z:=[z_1,\; z_2,\; \cdots,\; z_m]^T$, a time interval $T = [t_{\min}, t_{\max}].$\\
    \textbf{Output:} An optimal pole-interval $[c^*, d^*],$ Zolotarev poles $\{\sigma_i\}_{i=1}^n \subset [c^*, d^*],$ interpolation points $\{\vartheta_i\}_{i=1}^n \subset [0, +\infty)$, residue vector $\bm{\alpha}(t)\in\mathbb{R}^n$, the associated rational interpolant $r^{(c^*,d^*)}_{t,n}$ for each $t\in T.$
\begin{algorithmic}[1]
        \STATE Set the initial pole-interval as $[c_0, d_0] = [-n/(\sqrt{2}t_{\min}), -n/(\sqrt{2}t_{\max})]$.
        \STATE Find the minimizer $(c^*, d^*)$ of the problem~\eqref{minimize} using \texttt{fminsearch} with the initial guess $(c_0, d_0)$.
        \STATE Obtain the poles $\{\sigma_i\}_{i=1}^n$ and interpolation points $\{\vartheta_i\}_{i=1}^n$ from Algorithm~\ref{algo} with input $[c^*, d^*]$ and $n$.\\
        \STATE Form the matrix $\widehat{C} \in \mathbb{R}^{m\times n}$ with entries $\widehat{C}_{ij} = (z_i-\sigma_j)^{-1}$.\\
        \FOR{each $t \in T$,}
        \STATE Compute the residues $\bm{\alpha}(t)$ of $r^{(c^*,d^*)}_{t,n}$ via the linear system~\eqref{linsys}.\\
        \STATE Evaluate the approximant as $r^{(c^*,d^*)}_{t,n}(Z)=\widehat{C} \bm{\alpha}(t)$.\\
        \ENDFOR
    \end{algorithmic}
\end{algorithm}

Minimizing the discrete time-uniform error in~\eqref{minimize} over $Z$ does not in general guarantee that the target continuous time-uniform error in~\eqref{time_uni_err} will be bounded from above by a given error tolerance, as the latter is defined over the whole nonnegative real axis. However, this can be ensured in our setting, as the approximant $r^{(c^*, d^*)}_{t,n}(z)$ is uniformly continuous on $z\geq 0$ for all $t\in T.$

\begin{theorem}\label{thm_discont_cont}
    Let $Z:=[z_1,z_2,\cdots , z_m]^T$ be a discretization of $[0, +\infty)$ satisfying $0 \leq z_1<z_2<\cdots <z_m = L$.  For each $t \in T:=[t_{\min}, t_{\max}]$, let $r^*_{t,n}(z)=r^{(c^*,d^*)}_{t,n}(z)$ be the degree $n$ rational interpolant of $f_t(z) = \exp(-tz)$ with poles and interpolation points obtained from Algorithm~\ref{algo2}. Then we have 
    \begin{equation}\label{disc_errbnd}
    \max_{t\in T}\max_{z \geq 0}\vert r^*_{t,n}(z)-f_t(z)\vert \leq \max_{t\in T}\max_{z\in Z}\vert r^*_{t,n}(z)-f_t(z) \vert+\epsilon(\delta, L),
    \end{equation}
    where $\delta $ is the \enquote*{density} of $Z$ in the interval $[0, L]$ defined as
    \begin{equation}\label{density}
        \delta = \max_{x\in [0, L]}
        \min_{y\in Z} \vert x-y \vert,
    \end{equation}
    and $\epsilon(\delta, L)$ in the last term of~\eqref{disc_errbnd} satisfies
    \begin{equation}
        \lim_{\delta\rightarrow0}\lim_{L\rightarrow+\infty}\epsilon(\delta, L) = 0.
    \end{equation}
\begin{proof}
        For each $t \in T,$ we can use \cite[p.~86, Lem.~2]{Cheney} to show that 
        \begin{equation}\label{firstpart}
            \max_{z\in[0,L]}\left| r^*_{t,n}(z)-f_t(z) \right| \leq \max_{z\in Z}\left| r^*_{t,n}(z)-f_t(z) \right| + w_{\left(r^*_{t,n}-f_t\right)}(\delta),
        \end{equation}
        where 
        \begin{equation}
            w_{\left(r^*_{t,n}-f_t\right)}(\delta) = \max_{\substack{0\leq x,y \leq L\\ \left| x-y \right| \leq \delta}}\left| \left(r^*_{t,n}-f_t\right)(x)-\left(r^*_{t,n}-f_t\right)(y)\right|
        \end{equation}
        is the modulus of continuity of the function $r^*_{t,n}-f_t$. Denote the residues of $r^*_{t,n}$ as $\{\alpha^*_j(t)\}_{j=0}^n$ and the poles as $\{\sigma^*_j\}_{j=1}^n\subset (-\infty, 0)$ in the partial fraction form of~\eqref{partial_fraction_form}, then for any $s \geq L,$ we have 
        \begin{align*}
            \left|r^*_{t,n}(L)-r^*_{t,n}(s) \right| &\leq \sum_{j=1}^n \left| \alpha^*_j(t)\right| \left| \frac{1}{L-\sigma^*_j}-\frac{1}{s-\sigma^*_j}\right|\\
            &\leq \sum_{j=1}^n\left| \alpha^*_j(t)\right|\left( \frac{1}{L-\sigma^*_j}+\frac{1}{s-\sigma^*_j}\right) \\
            &\leq 2\sum_{j=1}^n\frac{\left| \alpha^*_j(t)\right| }{L-\sigma^*_j} \\
            &\leq \frac{2}{L}\sum_{j=1}^n\left| \alpha^*_j(t)\right|,
        \end{align*}
        and hence,
        \begin{align*}
        \bigl|| r^*_{t,n}(s)-f_t(s)| - | r^*_{t,n}(L)-f_t(L)| \bigr| &\leq \left| f_t(s)-f_t(L)+r^*_{t,n}(L)-r^*_{t,n}(s) \right| \\
        & \leq \left| f_t(s)-f_t(L)\right|+\left|r^*_{t,n}(L)-r^*_{t,n}(s) \right| \\
        & \leq f_t(L)+\left|r^*_{t,n}(L)-r^*_{t,n}(s) \right| \\
        &\leq f_t(L)+\frac{2}{L}\sum_{j=1}^n\left| \alpha^*_j(t)\right|.
        \end{align*}
    Therefore, 
    \begin{align}\label{secondpart}
            \max_{z\in [L, +\infty)}\left| r^*_{t,n}(z)-f_t(z) \right| &\leq \left| r^*_{t,n}(L)-f_t(L) \right| +f_t(L)+\frac{2}{L}\sum_{j=1}^n\left| \alpha^*_j(t)\right| \notag \\
            &\leq \max_{z\in Z}\left| r^*_{t,n}(z)-f_t(z) \right| +f_t(L)+\frac{2}{L}\sum_{j=1}^n\left| \alpha^*_j(t)\right|.
    \end{align}
    Combining~\eqref{secondpart} with~\eqref{firstpart}, we conclude
    \begin{align*}
         \max_{z\geq 0}\left| r^*_{t,n}(z)-f_t(z) \right| &=  \max\left\{\max_{z\in [0, L]}\left| r^*_{t,n}(z)-f_t(z) \right|,  \max_{z\in [L, +\infty)}\left| r^*_{t,n}(z)-f_t(z) \right|\right\} \\
         &\leq \max_{z\in Z}\left| r^*_{t,n}(z)-f_t(z) \right| + \widehat{\epsilon}(\delta, L),
    \end{align*}
    where 
    \begin{equation}
    \widehat{\epsilon}(\delta, L) = w_{\left(r^*_{t,n}-f_t\right)}(\delta)+f_t(L)+\frac{2}{L}\sum_{j=1}^n\left| \alpha^*_j(t)\right|. 
    \end{equation}
    Finally, by taking the maximum error over the time domain $T$, we have
    \begin{align}
    \max_{t\in T}\max_{z \geq 0}\left| r^*_{t,n}(z)-f_t(z)\right| \leq \max_{t\in T}\max_{z\in Z}\left| r^*_{t,n}(z)-f_t(z) \right|+\epsilon(\delta, L),
    \end{align}
    where 
    \begin{equation}\label{remainder}
    \epsilon(\delta, L) = \max_{t\in T}\left(w_{\left(r^*_{t,n}-f_t\right)}(\delta)\right)+f_{t_{min}}(L)+\frac{2}{L}\max_{t\in T}\sum_{j=1}^n\left| \alpha^*_j(t)\right|.
     \end{equation}    
    Since both $f_t(z)$ and $r^*_{t,n}(z)$ are uniformly continuous on $z\in[0,+\infty)$ for all $t\in T$, the first term on the right-hand side of~\eqref{remainder} vanishes as $\delta\to0$. The last two terms vanish as $L\to +\infty$. 
\end{proof}

   
\end{theorem}

\section{Floating-point arithmetic  analysis}\label{error_anal}
So far, we have considered only the scalar time-uniform error, the upper bound in~\eqref{time_uni_err}. In this section, let us discuss several sources of numerical errors when evaluating $\{r_{t,n}(A)\bm{b}\}_{t\in T}$ in floating-point arithmetic. 
For this analysis, we will consider the more general  case, often arising with finite element discretizations, where $A = M^{-1}K$ and $\bm{b} = M^{-1}\bm{q}$ for a symmetric positive definite matrix $M$ (mass matrix), a symmetric positive semidefinite matrix $K$ (stiffness matrix), and a vector $\bm{q}$. Now each shifted linear system 
$
(A- \sigma_iI)\bm{x}_i = \bm{b}
$
is equivalent to 
\begin{equation}
(K-\sigma_i M)\bm{x}_i = \bm{q}, \label{CMequ}
\end{equation}
which is still symmetric positive definite for all $\sigma_i < 0$. 
This reduces to the standard case when $M = I$, the identity matrix. Note that $A=M^{-1} K$ is similar to $K = M^{1/2} A M^{-1/2}$ and an equivalent upper bound to \eqref{time_uni_err} on the rational approximation error is available when the $2$-norm is replaced by the $M$-norm, $\|\bm x\|_M = \sqrt{{\bm x}^T M \bm x}$.

Throughout this section, we use  the standard model and notation of floating-point arithmetic; see \cite{higham2002accuracy}. We use  $fl(\cdot)$ to indicate that an expression is evaluated in floating-point arithmetic, and $u$ denotes the unit round-off.   In particular, $\widetilde x = fl(x) = x (1+\delta),$ where $|\delta|\leq u.$ Computed quantities are  labeled with a hat or tilde.

Given a rational approximant in partial fraction form, 
\begin{equation}\label{exact_appr}
    \bm{y}(t):=r_{t,n}(A)\bm{b} = \sum_{i=1}^n\alpha_i(t)\bm{x}_i,
\end{equation}
with each $\bm{x}_i$ being the exact solution to \eqref{CMequ}, we wish to study the numerically evaluated counterpart $fl({\bm{y}}(t))$ and bound the  distance $\left\|fl({\bm{y}}(t))-\bm{y}(t)\right\|_M$. We will study three sources of rounding errors, each in a separate subsection below. Each subsection bounds one error source, $E_1,E_2,E_3$, and by the triangle inequality,
\[
        \left\|fl({\bm{y}}(t))-\bm{y}(t)\right\|_M \leq E_1 + E_2 + E_3.
\]

\subsection{Error source 1: inaccurate linear solves}\label{error1}

The first source of error arises from the inaccuracy of solving the shifted linear systems~\eqref{CMequ}. 
Let $\widehat{\bm{x}}_i$ be the inexact solutions, and denote the associated residual vectors by 
\begin{equation}
\bm{e}_i:= (K-\sigma_iM)\widehat{\bm{x}}_i -\bm{q}.
\label{eq:resvecs}
\end{equation}
Consider the computed $\widehat{\bm{y}}(t)$ given by
\begin{equation}\label{lin_comb}
\widehat{\bm{y}}(t):=\sum_{i=1}^n\alpha_i(t)\widehat{\bm{x}}_i.
\end{equation}
Theorem~\ref{err1-bnd} below provides an upper bound for the resulting error. Importantly, this bound does not depend on the exact solutions $\bm{x}_i$, which are not available in practice.

\begin{theorem}\label{err1-bnd}
    The $M$-norm error of $\widehat{\bm{y}}$ defined in \eqref{lin_comb} satisfies 
    \[
    E_1 := \left\| \widehat{\bm{y}}(t)-\bm{y}(t)\right\|_M \leq \displaystyle\sum_{i=1}^n \frac{\left|\alpha_i(t)\right|}{\left|\sigma_i \right|} \left\| M^{-1}\bm{e}_i \right\|_M,
    \]
    where each $\bm{e}_i$ defined in \eqref{eq:resvecs}  is the residual associated with the $i$-th shifted linear system~\eqref{CMequ}.
\end{theorem}
\begin{proof} Applying the triangle inequality and using the symmetric positive definiteness of $M$, we have
    \begin{align*}
        \left\| \widehat{\bm{y}}(t)-\bm{y}(t)\right\|_M &= \left\| \sum_{i=1}^n \alpha_i(t)(\widehat{\bm{x}}_i-\bm{x}_i) \right\|_M \\
        &\leq \sum_{i=1}^n \left|\alpha_i(t)\right| \left\| \widehat{\bm{x}}_i-\bm{x}_i \right\|_M \\
        &\leq \sum_{i=1}^n \left|\alpha_i(t)\right| \left\|(M^{-1}K-\sigma_iI)^{-1} \right\|_M \left\|(M^{-1}K-\sigma_iI)(\widehat{\bm{x}}_i-\bm{x}_i) \right\|_M \\
        & =  \sum_{i=1}^n \left|\alpha_i(t)\right| \left\| (M^{-1/2}KM^{-1/2}  - \sigma_iI)^{-1} \right\|_2 \left\|M^{-1}\bm{e}_i \right\|_M\\
        &\leq \sum_{i=1}^n \frac{\left|\alpha_i(t)\right|}{\left| \sigma_i \right|} \left\|M^{-1}\bm{e}_i \right\|_M,
    \end{align*}
as asserted. We note that 
\[
\left\|M^{-1}\bm{e}_i \right\|_M \leq 
\frac{1}{\sqrt{\lambda_{\min}(M)}} \left\|\bm{e}_i \right\|_2 \leq 
\frac{1}{{\lambda_{\min}(M)}} \left\|\bm{e}_i \right\|_M,
\]
which could be used to remove the dependence on $M^{-1}$ if (an estimate for) $\lambda_{\min}(M)$ is available.
\end{proof}

Assuming that each shifted linear system is solved using a direct method such as sparse Cholesky factorization~\cite[Chapter~11]{golub2013matrix}, the norm of each residual $\bm{e}_i$ is expected to be on the order of machine precision, and therefore the error is expected to be small provided that the residues $\alpha_i(t)$ are of small to moderate magnitude. However, $\alpha_i(t)$ may become large if the Cauchy matrix involved in the linear system~\eqref{linsys} has small singular values. This happens in particular when poles~$\sigma_i$ are close together.

\subsection{Error source 2: inaccurate residues}\label{error2}
Recall that solving~\eqref{linsys} using  standard double precision with $16$ digits of accuracy suffers from the increase of the condition number of the Cauchy matrix for large degrees of approximation. To avoid this, we recommend using multiprecision (e.g., with $128$ digits of accuracy). This is computationally tractable due to the moderate size of the Cauchy matrix. 
On the other hand, we do not wish to form the linear combination~\eqref{lin_comb} with residues $\alpha_i(t)$ in  multiprecision, since the $\widehat{\bm{x}}_i$ may be very high-dimensional vectors. We therefore round the residues to double precision, introducing an additional error in the computation of~\eqref{lin_comb}. Theorem~\ref{round_err_resid} below gives an upper bound for this error.

\begin{theorem}\label{round_err_resid}
    Let $\widehat{\bm{y}}(t)$ be the linear combination~\eqref{lin_comb} formed with the exact residues~$\alpha_i(t)$. Let $\widetilde{\alpha}_i(t)$ be the residues rounded to double precision, and define 
    \begin{align}\label{numer_lincomb}
    \widetilde{\bm{y}}(t):=\sum_{i=1}^n\widetilde{\alpha}_i(t)\widehat{\bm{x}}_i. 
    \end{align}
    Then we have 
    \begin{align}
    E_2 := \left\| \widetilde{\bm{y}}(t) - \widehat{\bm{y}}(t) \right\|_M \leq u \sum_{i=1}^n \left|\alpha_i(t)\right|\left\| \widehat{\bm{x}}_i\right\|_M. \label{round_errbnd}
    \end{align} 
    \begin{proof}
    Applying the triangular inequality, we have
        \[
        \left\| \widetilde{\bm{y}}(t) - \widehat{\bm{y}}(t) \right\|_M \leq \sum_{i=1}^n \left|\widetilde{\alpha}_i(t)-\alpha_i(t)\right| \left\|\widehat{\bm{x}}_i\right\|_M \leq u \sum_{i=1}^n \left|\alpha_i(t)\right| \left\|\widehat{\bm{x}}_i\right\|_M .
        \]
    \end{proof}
\end{theorem}


Similarly to the numerical error of the computed shifted linear systems, the bound in Theorem~\ref{round_err_resid} suggests that the error arising from rounding the residues may be amplified when the residues are of very large magnitude.

\subsection{Error source 3: inaccurate linear combination}\label{error3}
The third source of error is the inaccuracy from forming the linear combination~\eqref{numer_lincomb} given the computed residues $\widetilde{\alpha}_i(t)$ and solutions $\widehat{\bm{x}}_i$. Theorem~\ref{err2_bnd} below provides an upper bound for this error, which also depends on the magnitude of residues.
\begin{theorem}\label{err2_bnd}
    Let $fl\left(\widetilde{\bm{y}}(t)\right)$ be the computed result of $\widetilde{\bm{y}}(t)$ in~\eqref{lin_comb}, then we have
    \[
    E_3 := \left\| fl\left(\widetilde{\bm{y}}(t)\right)-\widetilde{\bm{y}}(t) \right\|_M 
    \leq 
    \gamma_n \sqrt{\kappa_2(M)}\,\Big\|  \big|\widehat{X}\big|\cdot |\widetilde{\bm{\alpha}}(t)|\Big\|_M
    \]
    where $\widehat{X}:=[\widehat{x}_1, \widehat{x}_2,\cdots, \widehat{x}_n]$, $\widetilde{\bm{\alpha}}(t): = [\widetilde{\alpha}_1(t), \widetilde{\alpha}_2(t), \cdots, \widetilde{\alpha}_n(t)]^T$, $\kappa_2(M)$ is the $2$-norm condition number of $M$, and $\gamma_n := \frac{nu}{1-nu}$ ($u$ is the unit roundoff). Here, $|\cdot|$ denotes the entrywise absolute value of a vector or matrix.
\end{theorem}
\begin{proof}
    Applying the results in~\cite[Section~3.5]{higham2002accuracy}, we have
     \[
   \left| fl\left(\widetilde{\bm{y}}(t)\right)-\widetilde{\bm{y}}(t) \right| \leq \gamma_n \left|\widehat{X}\right|\cdot \left|\widetilde{\bm{\alpha}}(t)\right|
    \]
    where the inequality holds entrywise. Therefore, it follows that  
     \[
   \left\| fl\left(\widetilde{\bm{y}}(t)\right)-\widetilde{\bm{y}}(t) \right\|_2 \leq \left\|\gamma_n \left|\widehat{X}\right|\cdot |\widetilde{\bm{\alpha}}(t)|\right\|_2, 
    \]
    and hence 
    \begin{align*}
    \left\| fl\left(\widetilde{\bm{y}}(t)\right)-\widetilde{\bm{y}}(t) \right\|_M &\leq \sqrt{\lambda_{\max}(M)}  \left\| fl\left(\widetilde{\bm{y}}(t)\right)-\widetilde{\bm{y}}(t) \right\|_2 \\ &\leq \sqrt{\lambda_{\max}(M)}\left\|\gamma_n \left|\widehat{X}\right|\cdot \left|\widetilde{\bm{\alpha}}(t)\right|\right\|_2 \\ &\leq \sqrt{\frac{\lambda_{\max}(M)}{\lambda_{\min}(M)}}\left\|\gamma_n \left|\widehat{X}\right|\cdot \left|\widetilde{\bm{\alpha}}(t)\right|\right\|_M \\ & = \sqrt{\kappa_2(M)}\left\|\gamma_n \left|\widehat{X}\right|\cdot \left|\widetilde{\bm{\alpha}}(t)\right|\right\|_M.
    \end{align*}
\end{proof}

Finally, note that if no other floating-point errors occur when evaluating the rational approximant~\eqref{exact_appr}, $fl\left(\widetilde{\bm{y}}(t)\right)$ is indeed the final numerical result of the evaluation and we may write $fl\left(\widetilde{\bm{y}}(t)\right) = fl(\bm{y}(t))$.


\medskip

\subsection{Growth of the residues}

Let us briefly comment on the growth of the residues $\bm{\alpha}(t)$ computed via solving~\eqref{linsys} involving the $n\times n$ Cauchy matrix~$C$.
We have 
$$
    \| \bm{\alpha}(t)\|_\infty \leq \| \bm{\alpha}(t)\|_2 =\| C^{-1} \bm{f}(t) \|_2 
    \leq \|\bm{f}(t)\|_2/\sigma_{n}(C).
$$
Clearly, $\|\bm{f}(t)\|_2\leq \sqrt{n}$ with $n$ interpolation points for the exponential function. 

The decay properties of the singular values of Cauchy (and other matrices with displacement structure) have been well studied~\cite{beckermann2017singular}. In particular, it is known that 
\[
\sigma_n(C) \leq \sigma_1(C) Z_{n-1}([c,d],[0,+\infty]),
\]
where $Z_{n-1}$ is the Zolotarev number of order $n-1$ of the condenser $([c,d],[0,+\infty])$. Moreover, it is known that these numbers satisfy
\[
\lim_{n\to +\infty}  \left(Z_{n}([c,d],[0,+\infty])\right)^{1/n}=\exp\left(-\frac{1}{\operatorname{cap}([c,d],[0,+\infty])}\right),
\]
where $\operatorname{cap}$ denotes the logarithmic capacity of a condenser.

In our setting, it is easy to get an upper bound on $\sigma_1(C)$ since all entries of $C$  are bounded by $1/|d|$. By H{\"o}lder's matrix inequality $\|C\|_2 \leq \sqrt{\|C\|_1 \|C\|_\infty}$, and the fact that
\[
\|C\|_1
=
\max_{1\le j\le n}
\sum_{i=1}^n |C_{ij}|
\le
\frac{n}{|d|}, \quad 
\|C\|_\infty
=
\max_{1\le i\le n}
\sum_{j=1}^n |C_{ij}|
\le
\frac{n}{|d|},
\]
we have
\[
\sigma_1(C) = \|C\|_2 \leq \frac{n}{|d|}.
\]
Unfortunately, the above results do only provide us with an exponentially decaying  \emph{upper} bound  on the smallest singular value $\sigma_n(C)$. In order to obtain an upper bound the growth of $\bm{\alpha}(t)$, we would need a \emph{lower} bound on $\sigma_n(C)$. This bound would necessarily need to take into account that 
$C$ is generated from optimal Zolotarev interpolation nodes and poles on the condenser $([c^*,d^*],[0,+\infty))$. We are currently not aware of such a lower bound on the singular values of Cauchy matrices with Zolotarev nodes and poles. But we observe numerically that $\sigma_n(C)$ indeed decays exponentially at the expected rate and \emph{not faster.} 

In any case, our situation is further complicated by the fact that the plate $[c^*,d^*]$ of the condenser is not fixed but varies with $n$. The dependence on $n$ is a complicated matter that is beyond the scope of this paper, and we leave this for future work.

\section{Numerical experiments}\label{exper}
This section is dedicated to numerical experiments. Section~\ref{conc_appr} illustrates the approximation using concentrated real poles proposed in section~\ref{sec:weight}. The remaining subsections concern the approximation using distinct real poles proposed in section~\ref{sec:zolo}. In particular, section~\ref{herm_persp} justifies the suitability of our proposed poles from the perspectives of Hermite integral formula. Section~\ref{err_anal_exper} presents an illustration of our error analysis from section~\ref{error_anal}. 
An application of our new poles to a transient electromagnetic (TEM) problem is provided in section~\ref{tem}, including a comparison to complex-valued RKFIT poles. Finally, section~\ref{teb} presents an example demonstrating that the proposed uniform rational approximation of the matrix exponential maintains its high accuracy over a wide range of time intervals. 

All algorithms are implemented in MATLAB R2024b, and all experiments are performed on a Windows laptop equipped with an Intel i5-13500H CPU running at 2.6~GHz and 32~GB of RAM. The code for reproducing the numerical experiments is available
at \texttt{\url{https://github.com/nla-group/uniform-in-time-exponential}}.

\subsection{Approximation with a weight function}\label{conc_appr}

We begin with an example of  concentrated real-pole approximation. 
Here we use the time interval $T = [10^{-3}, 1]$, and for ease of comparison the same $T$ will be used in most other experiments.

Figure~\ref{unweighted_weighted} shows the $L^\infty$-norm approximation error using the poles proposed in~\cite{guttel2025concentrated} which minimize the unweighted time-uniform error~\eqref{time_uni_err}, and the poles given by~\eqref{opt_poles} which minimize the weighted time-uniform error~\eqref{time_uni_err_weighted}. We clearly see that introducing the weight function sacrifices accuracy at early times but yields a much higher level of accuracy at later times. This choice of weighted poles will be particularly suited to applications where $\|\exp(-tA)\bm{b}\|_2$ decays as $t$ increases, and it is desirable to retain a high relative accuracy over the whole time interval~$T$.


\begin{figure}[H]
    \centering
    \begin{subfigure}[b]{0.48\textwidth}
        \hspace*{-3mm}\includegraphics[width=1.0\textwidth]{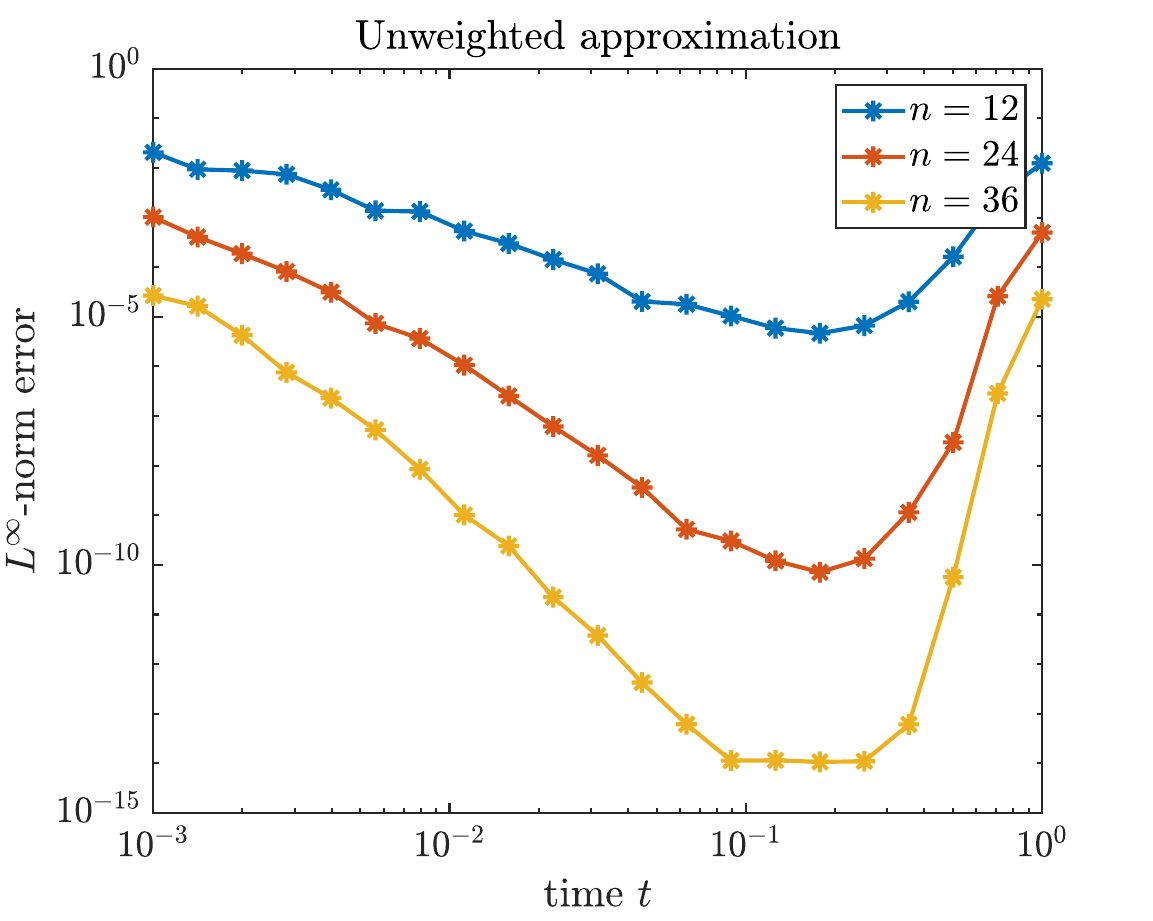}
    \end{subfigure}
    \begin{subfigure}[b]{0.48\textwidth}
        \includegraphics[width=1.0\textwidth]{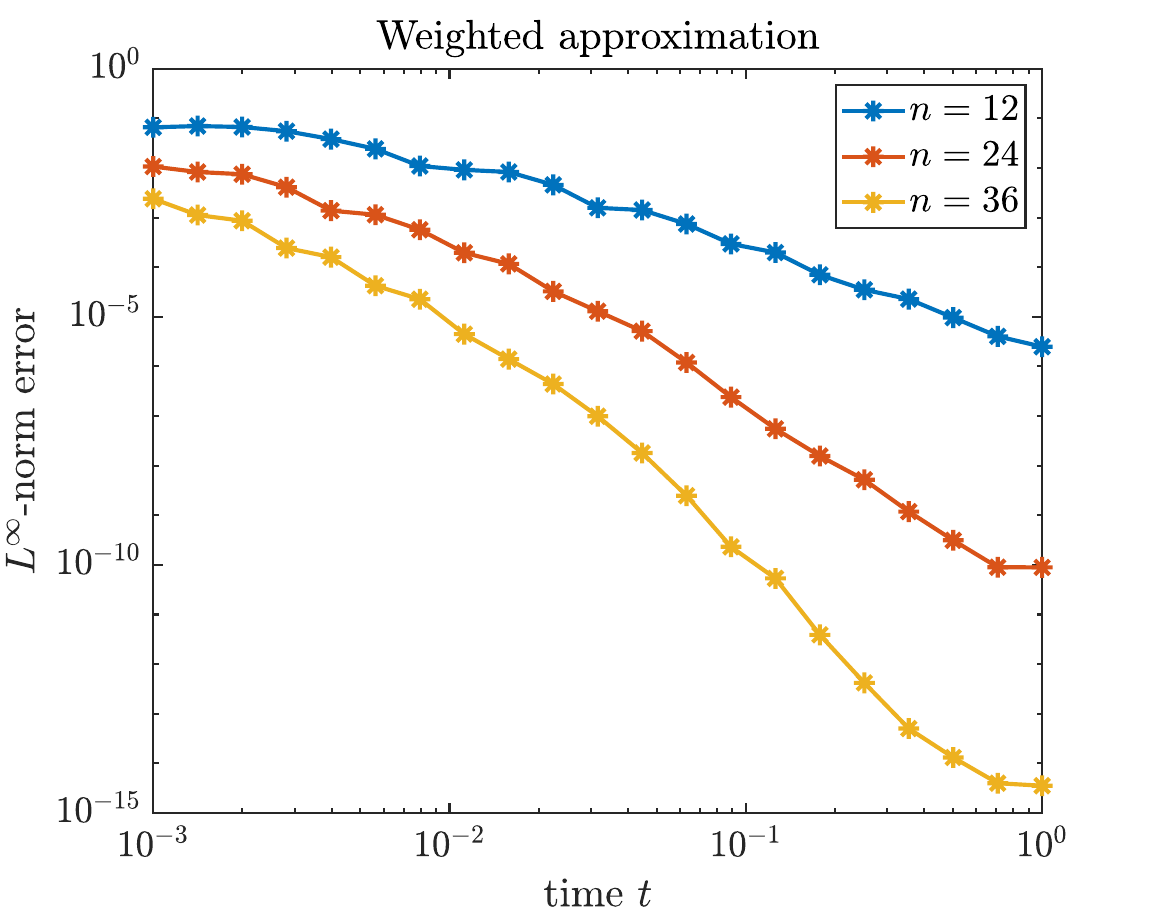}
    \end{subfigure}
    
    \caption{\textit{$L^{\infty}$-norm error using the unweighted optimal poles (left) and the weighted optimal poles (right).} }
    \label{unweighted_weighted}
\end{figure}

\subsection{Comparison with Zolotarev approximation}\label{herm_persp}
We now provide examples of approximations using distinct real poles computed by our Zolotarev  approach, illustrating the discussion at the beginning of section~\ref{sec:zolo}. 

Figure~\ref{contourplot} shows the level curves of $\log_{10} \vert \exp(-tz)/s_n(z)\vert$ with $t = t_{\min}=10^{-3}$ and $t = t_{\max}=1$ over the square domain 
$S = \{z \in \mathbb{C}:|\operatorname{Re}(z)| \le 10,\ |\operatorname{Im}(z)| \le 10\}$
for the degree $n = 25$ rational approximants computed using Algorithm~\ref{algo2}. In each plot, only five Zolotarev poles (black dots) and six interpolation points (red dots) are displayed, while the remaining poles and interpolation points are off-scale. For each choice of $t = t_{\min}$ and $t=t_{\max}$, there exists a contour $\Gamma$ where the order of $\vert \exp(-tz)/s_n(z)\vert$ is at least uniformly below $O(10^{-3}).$ In view of the Hermite integral error formula~\eqref{herm_err}, this illustrates the high accuracy of our rational approximants.

\begin{figure}[H]
    \centering
    \begin{subfigure}[b]{0.48\textwidth}
        \hspace*{-3mm}\includegraphics[width=1.0\textwidth]{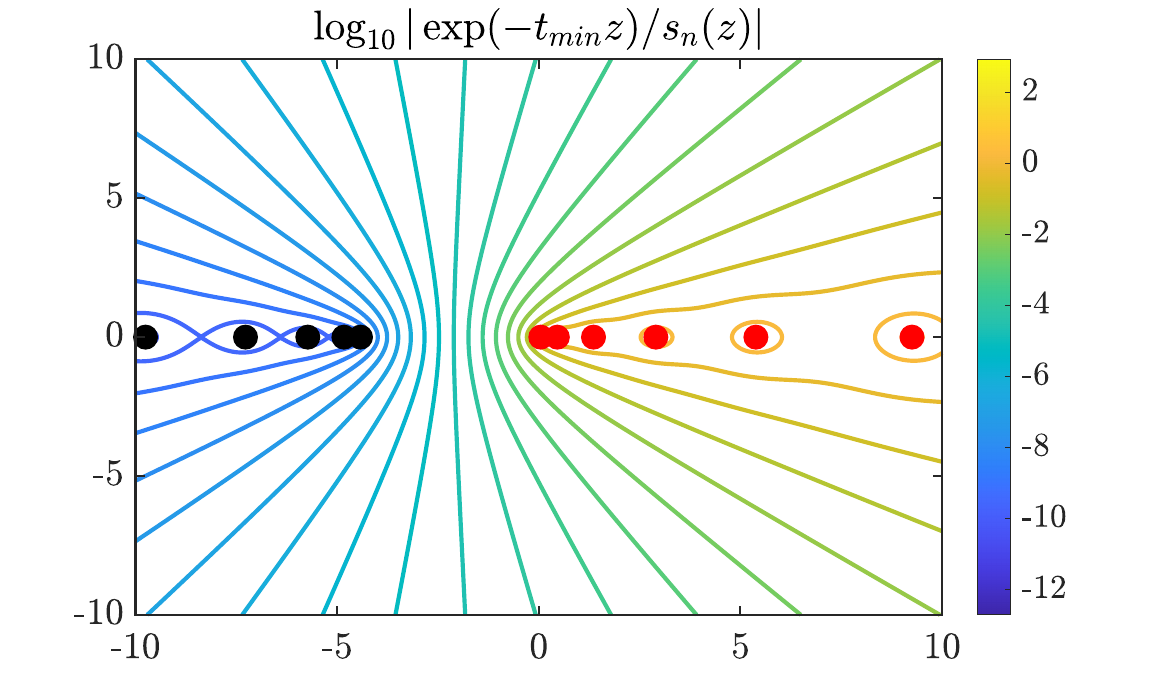}
    \end{subfigure}
    \begin{subfigure}[b]{0.48\textwidth}
        \includegraphics[width=1.0\textwidth]{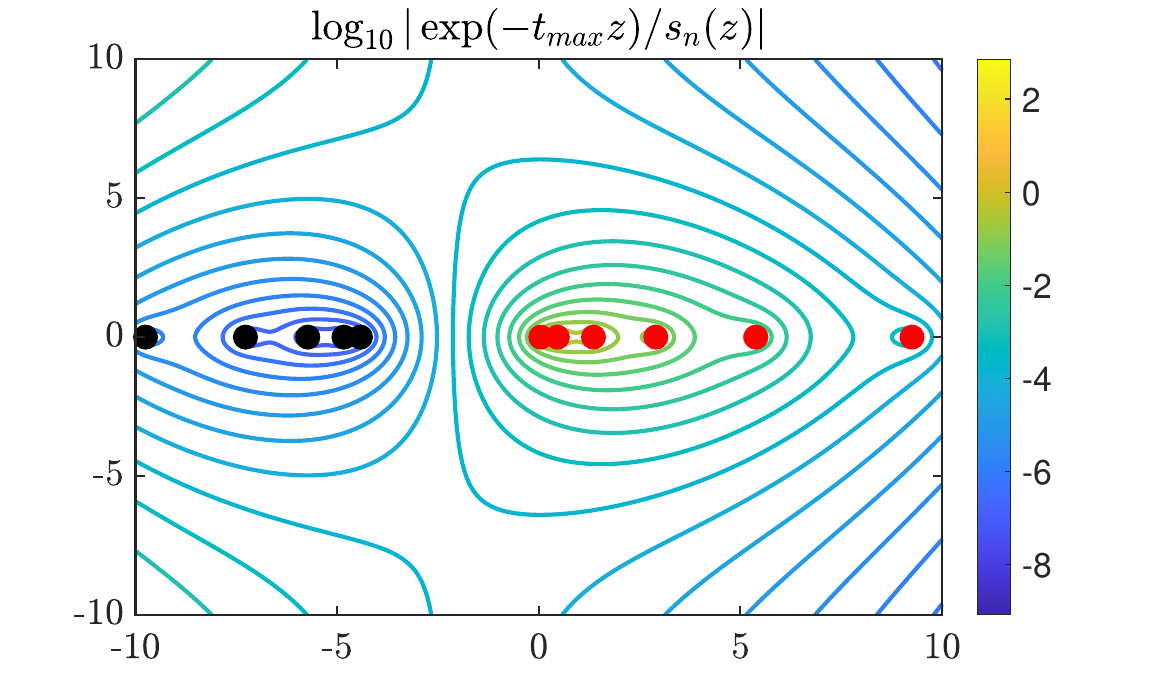}
    \end{subfigure}
    
    \caption{\textit{Level curves of the base-10 logarithm of $\vert \exp(-tz)/s_n(z)\vert$ for the degree $n = 25$ approximants computed using Algorithm~\ref{algo2}. Left: $t = t_{\min} = 10^{-3}$. Right: $t = t_{\max} = 1.$} }
    \label{contourplot}
\end{figure}

In Figure~\ref{three_types} we directly compare three different time-uniform approximants over $T=[10^{-3},1]$. We find that $33$ concentrated real poles from~\cite{guttel2025concentrated} are needed to ensure that the time-uniform error over $T$ is reduced to about $10^{-4}$, whereas the same can be achieved by using only $21$ Zolotarev real poles. This is due to the extra flexibility with which the Zolotarev poles are  placed on the negative real axis. A further advantage of Zolotarev poles is computational: distinct poles allow us to solve all the associated shifted linear systems in parallel. 

Figure~\ref{three_types} also includes a shared-pole family of RKFIT approximants with $18$ \emph{complex} poles~\cite{berljafa2017rkfit}. When evaluating these rational functions for matrix arguments, the number of required shifted linear systems is halved to $9$ since the poles appear in conjugate pairs. Nevertheless, the complex arithmetic involved can still lead to an overall longer time for solving these shifted linear systems, despite the smaller number of systems required. This will be demonstrated in section~\ref{tem}. 

\begin{figure}
    \centering
\includegraphics[width=0.7\linewidth]{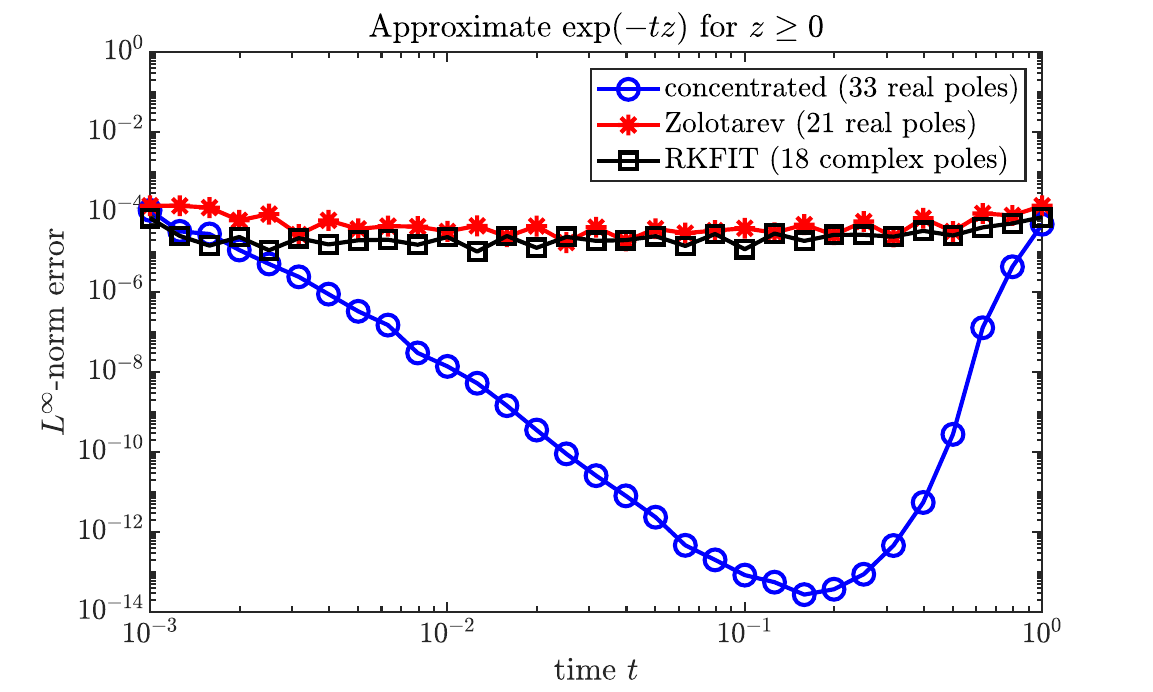}
    \caption{\textit{To achieve four-digit accuracy over the interval $T = [10^{-3}, 1]$, we need $33$ concentrated real poles, $21$ Zolotarev real poles, or $18$~complex RKFIT poles, respectively. }}
    \label{three_types}
\end{figure}

Finally, Figure~\ref{tue_all_T} shows the time-uniform error~\eqref{time_uni_err} obtained using the proposed Zolotarev rational interpolant, compared to those obtained using the optimal concentrated-pole approximants and the RKFIT approximants, for degrees $n$ from $1$ to $30$, and time intervals $T=[10^{-\beta},1]$ with $\beta = 1,\; 2,\; 3, \; 4$. For $\beta = 2,\; 3, \; 4$, the errors of the Zolotarev rational interpolants consistently lie between that of the optimal concentrated-pole approximants and that of the RKFIT approximants. An exception occurs for $\beta = 1$ where the Zolotarev interpolants and the optimal concentrated-pole approximant yield almost indistinguishable convergence. This is expected as for narrow time ranges (small values of $\beta$), we  approach the single-time approximation problem for which the poles of the best real-pole approximants coalesce \cite{borwein1983rational}. 

\begin{figure}
    \centering
    \begin{subfigure}[b]{0.48\textwidth}
        \hspace*{-3mm}\includegraphics[width=1.0\textwidth]{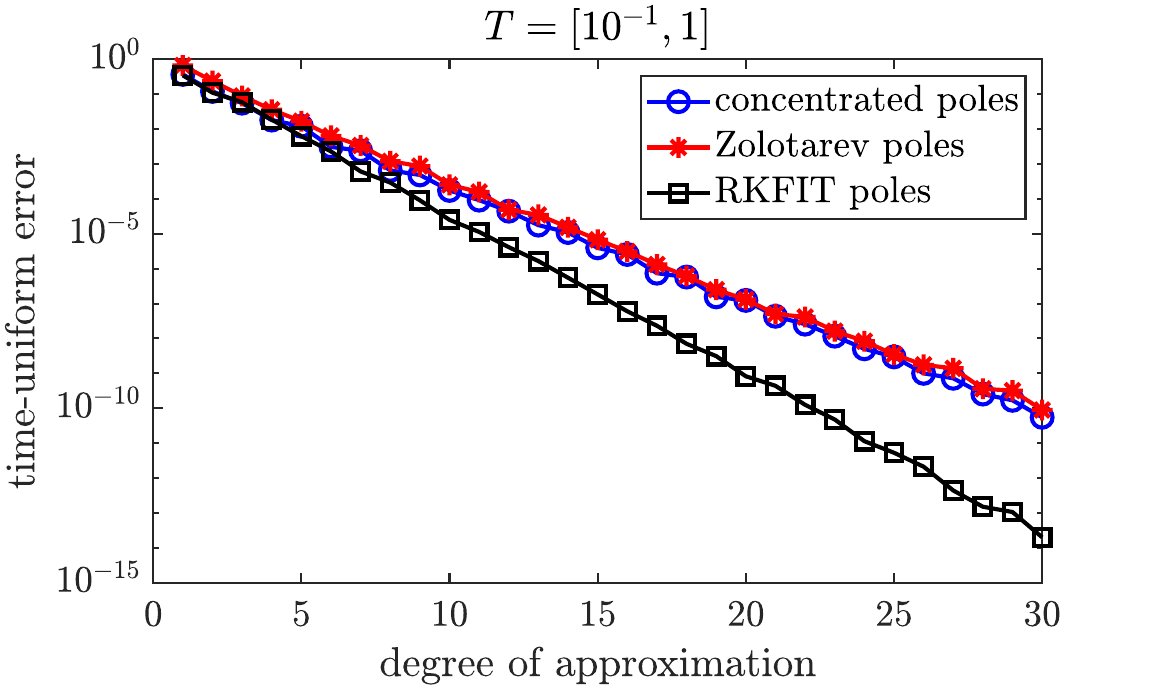}
    \end{subfigure}
    \begin{subfigure}[b]{0.48\textwidth}
    \includegraphics[width=1.0\textwidth]{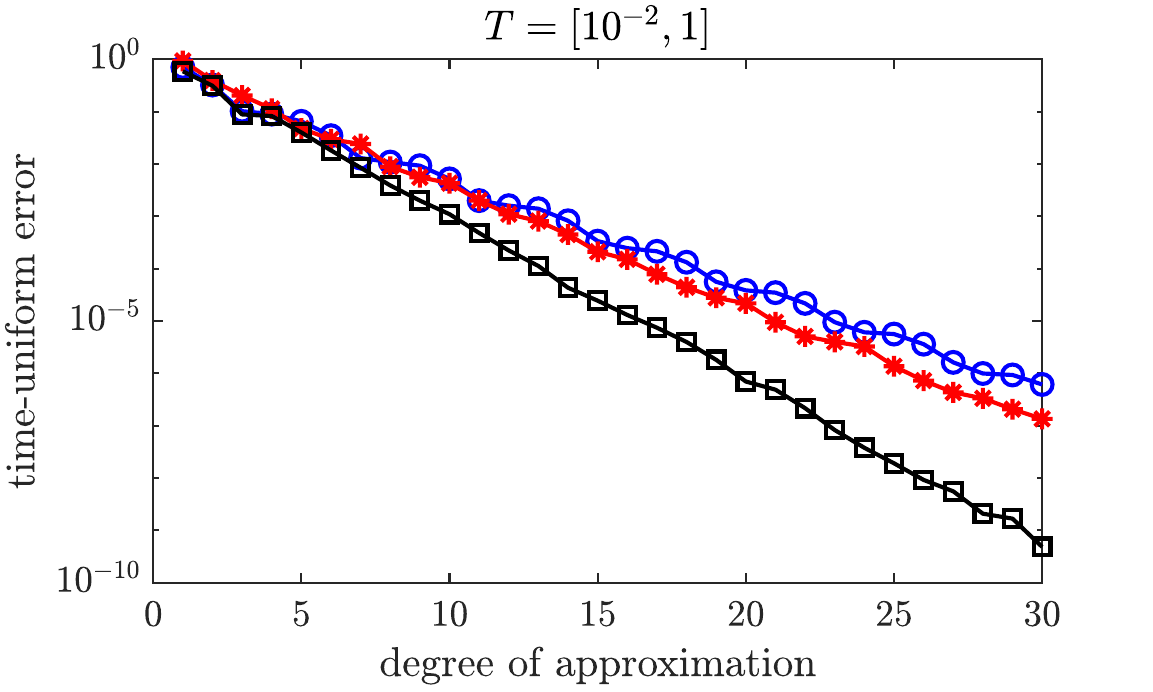}
    \end{subfigure}
    
    \vspace{0.5cm} 
    
    \begin{subfigure}[b]{0.48\textwidth}
         \hspace*{-3mm}\includegraphics[width=1.0\textwidth]{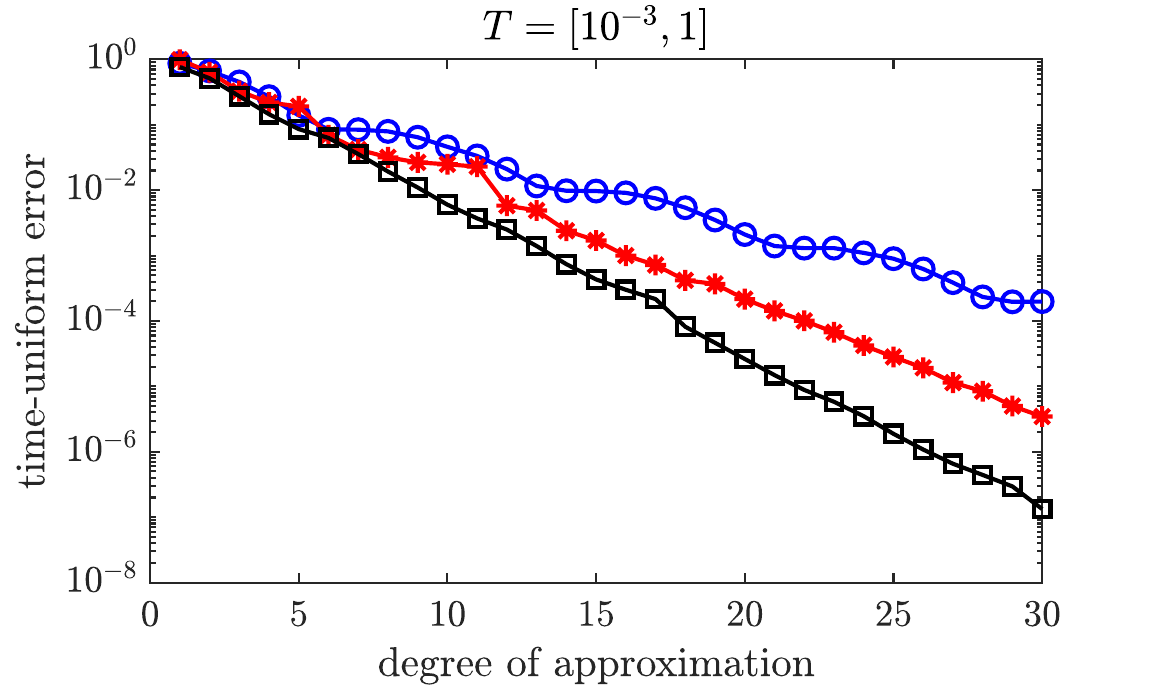}
    \end{subfigure}
    \begin{subfigure}[b]{0.48\textwidth}
        \includegraphics[width=1.0\textwidth]{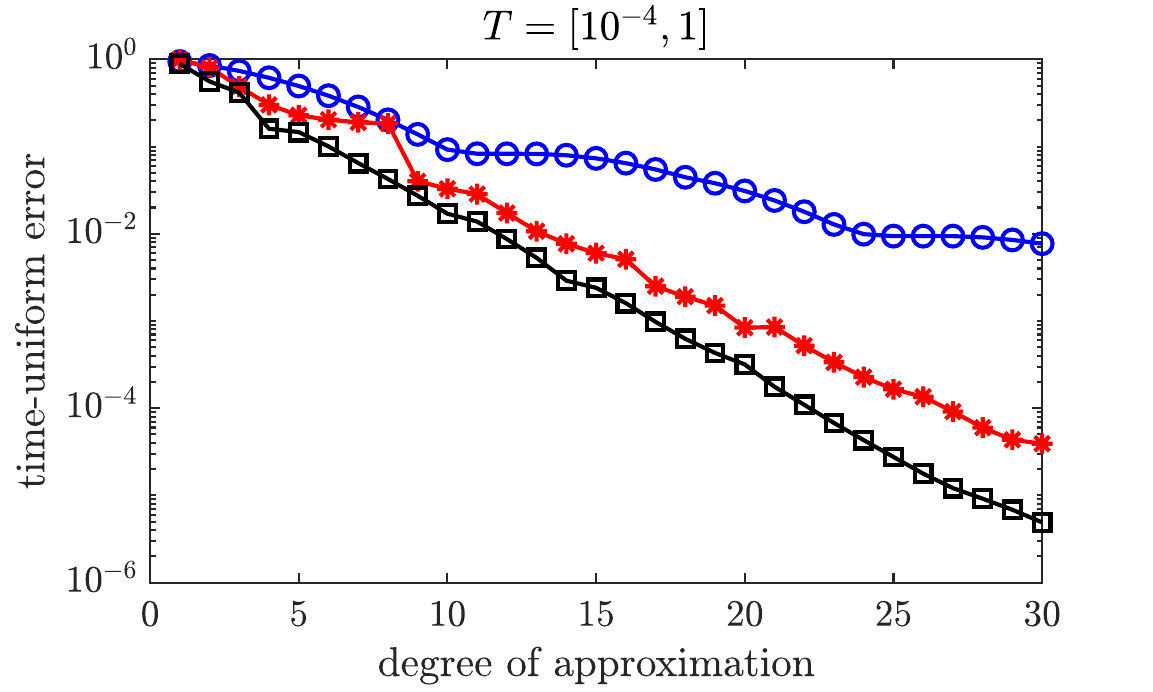}
    \end{subfigure}
    
    \caption{\textit{Time-uniform error obtained by the three types of approximants over different time intervals $T$, with the degrees $n=1,2,\ldots,30$.} }
    \label{tue_all_T}
\end{figure}

\subsection{Illustration of the error analysis}\label{err_anal_exper}
This section illustrates our floating-point error analysis in section~\ref{error_anal}. Figure~\ref{Stab_anal_plot} shows the numerically computed relative error (left-hand side of~\eqref{t_err_bnd}), its scalar error bound (right-hand side of~\eqref{t_err_bnd}),
\begin{eqnarray}
 \frac{\left\| r_{t,n}(A)\bm{b}-\exp(-tA)\bm{b}\right\|_2}{\left\|\bm{b}\right\|_2} 
   \leq \max_{z \geq 0}\left| r_{t,n}(z)-\exp(-tz)\right| \label{t_err_bnd}
\end{eqnarray}
along with the floating-point error bounds given by Theorem~\ref{err1-bnd},~\ref{round_err_resid},~\ref{err2_bnd} and their combination, for degree $n = 25,\, 35,\, 45$. In this example, we use a symmetric positive semidefinite tridiagonal matrix $A \in \mathbb{R}^{1,000 \times 1,000}$ and a normalized random vector~$\bm{b}$ with i.i.d.\ Gaussian entries.
We observe that the computed error decreases for $n$ between $25$ and $35$, during which the floating-point error bounds are effectively controlled. However, for $n$ between $35$ and $45$, the error no longer decreases because the combined floating-point error bound exceeds the scalar bound. In particular, for $n = 45$, where the floating-point error bounds fully dominate, the computed error exceeds its scalar bound as a consequence of this numerical instability.

\begin{figure}
    \centering
    \begin{subfigure}{0.32\textwidth}
        \centering
        \includegraphics[width=1.0\linewidth]{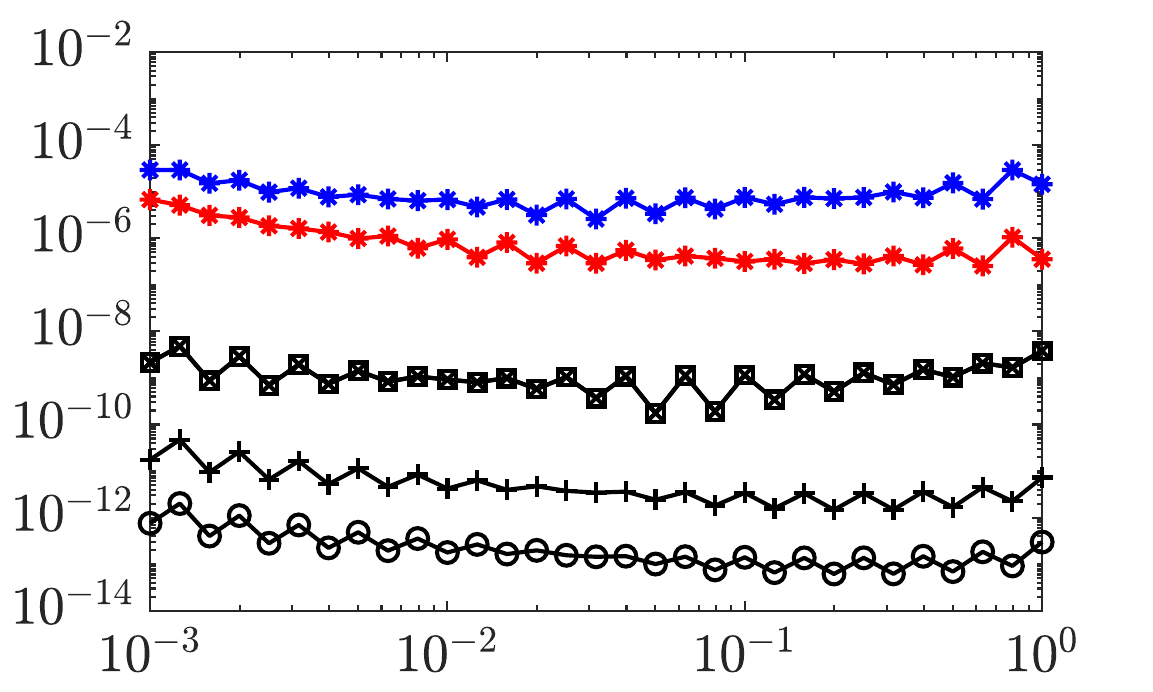}
        \caption{$n=25$}
    \end{subfigure}
    \hfill
    \begin{subfigure}{0.32\textwidth}
        \centering
        \includegraphics[width=1.0\linewidth]{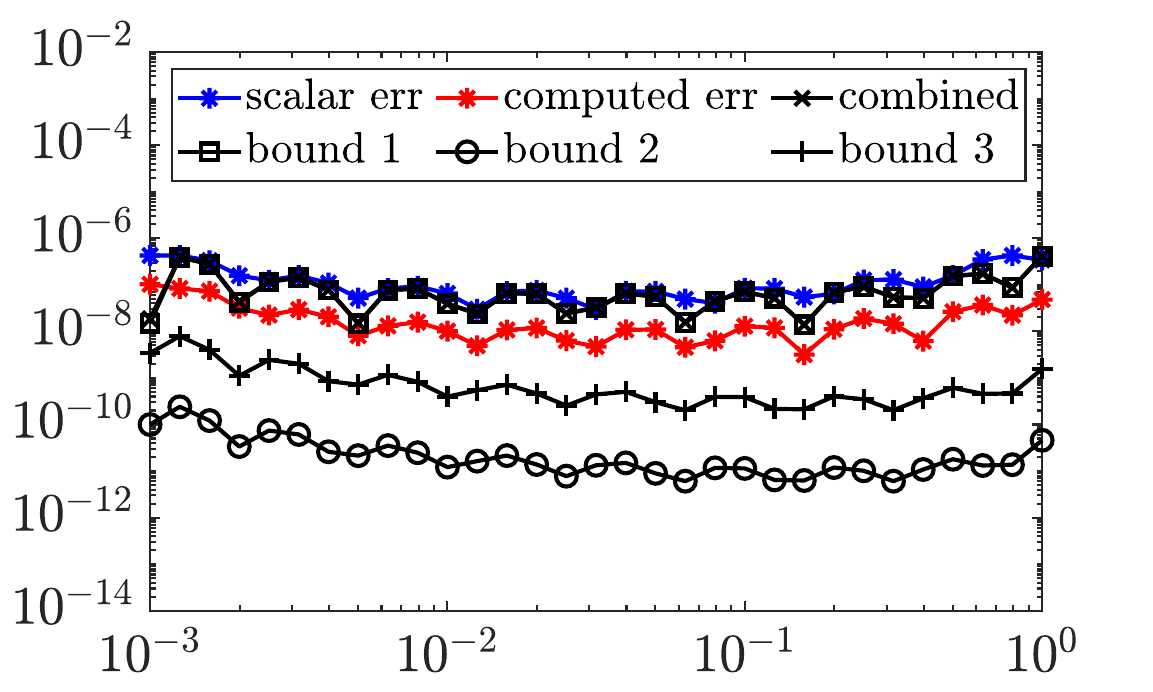}
        \caption{$n=35$}
    \end{subfigure}
    \hfill
    \begin{subfigure}{0.32\textwidth}
        \centering
        \includegraphics[width=1.0\linewidth]{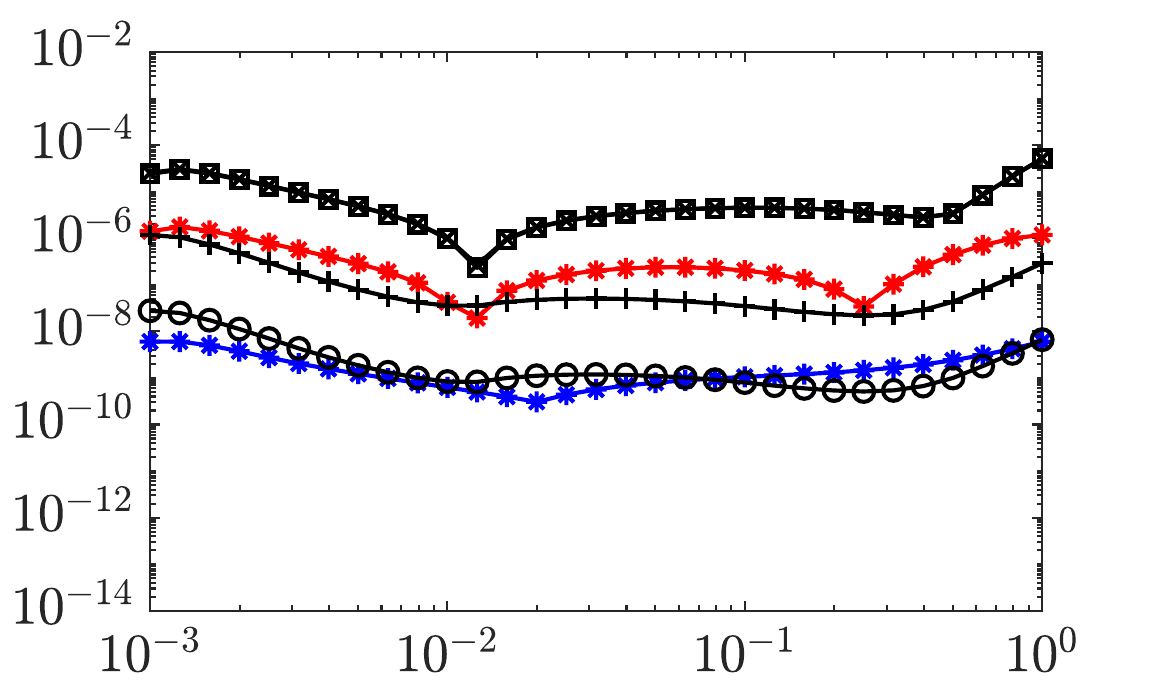}
        \caption{$n=45$}
    \end{subfigure}
    \caption{\emph{The computed error (red), the scalar bound (blue), and the floating-point error bounds derived in section~\ref{error_anal} (black) over the time interval $T=[10^{-3},1]$ for degrees $n=25,\, 35,\, 45$. We observe that the scalar error bound uniformly decreases as $n$ is increased, while the floating-point error increases.}}
    \label{Stab_anal_plot}
\end{figure}

\subsection{TEM problem}\label{tem}

An important application of time-uniform approximation of the matrix exponential arising in geophysics with the transient electromagnetic method (TEM). Recall from  Figure~\ref{three_types} that  $21$ real Zolotarev poles achieved about the same level of scalar time-uniform error as  $18$ RKFIT complex poles. The question then arises how factorizing and solving 21 real (positive definite) linear systems compares to solving $9$ complex-shifted linear systems. Figure~\ref{timing} compares the computation  time for solving these systems $(K-\sigma_i M)\bm{x}_i = \bm{q}$. Here,  both $K$ and $M$ are real sparse finite element matrices of size $N = 181,302$, corresponding to the \texttt{tem181302} electromagnetic problem from~\cite{borner2015three}; also available in the SuiteSparse Matrix Collection~\cite{davis2011university}. The real Zolotarev  shifts incur an average of $3.98$s per system, totaling $83.51$s, while the RKFIT complex shifts require an average of $34.67$s per system, totaling $311.99$s. So, in this case, the sequential linear solve speedup obtained with the Zolotarev approach is about 3.7-fold. If all linear solves are done in parallel, the parallel speedup with the Zolotarev approach is even 8.7-fold.

\begin{figure}
    \centering
    \begin{subfigure}[b]{0.48\textwidth}
        \hspace*{-3mm}\includegraphics[width=1.0\textwidth]{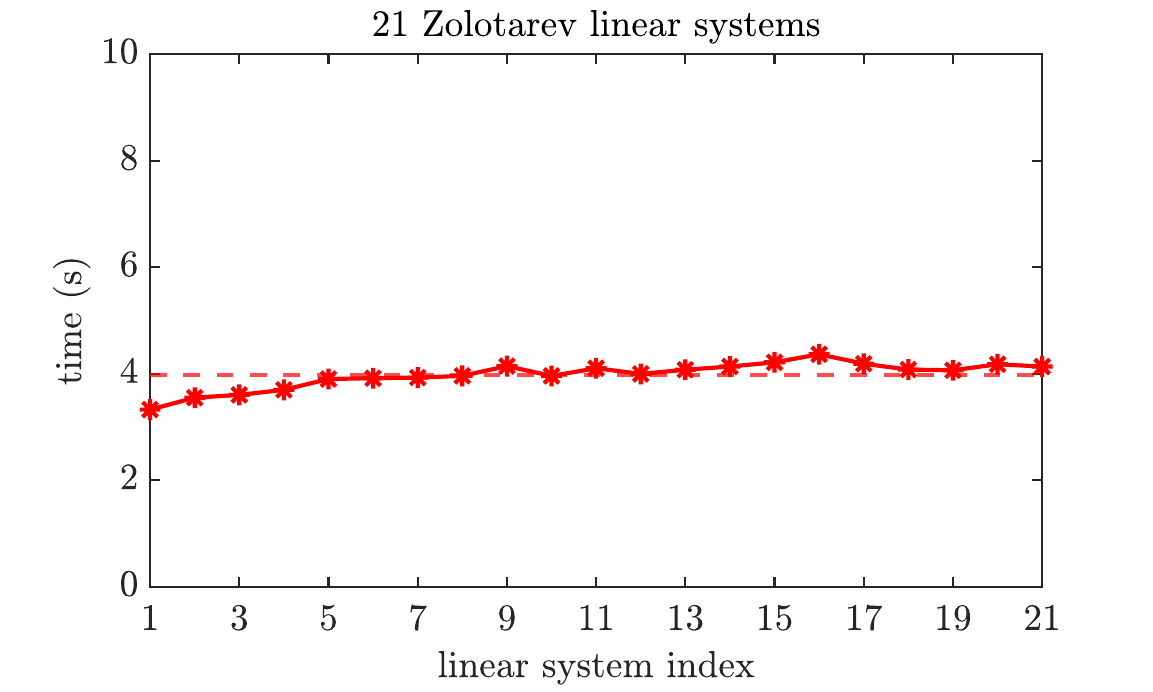}
    \end{subfigure}
    \begin{subfigure}[b]{0.48\textwidth}
        \includegraphics[width=1.0\textwidth]{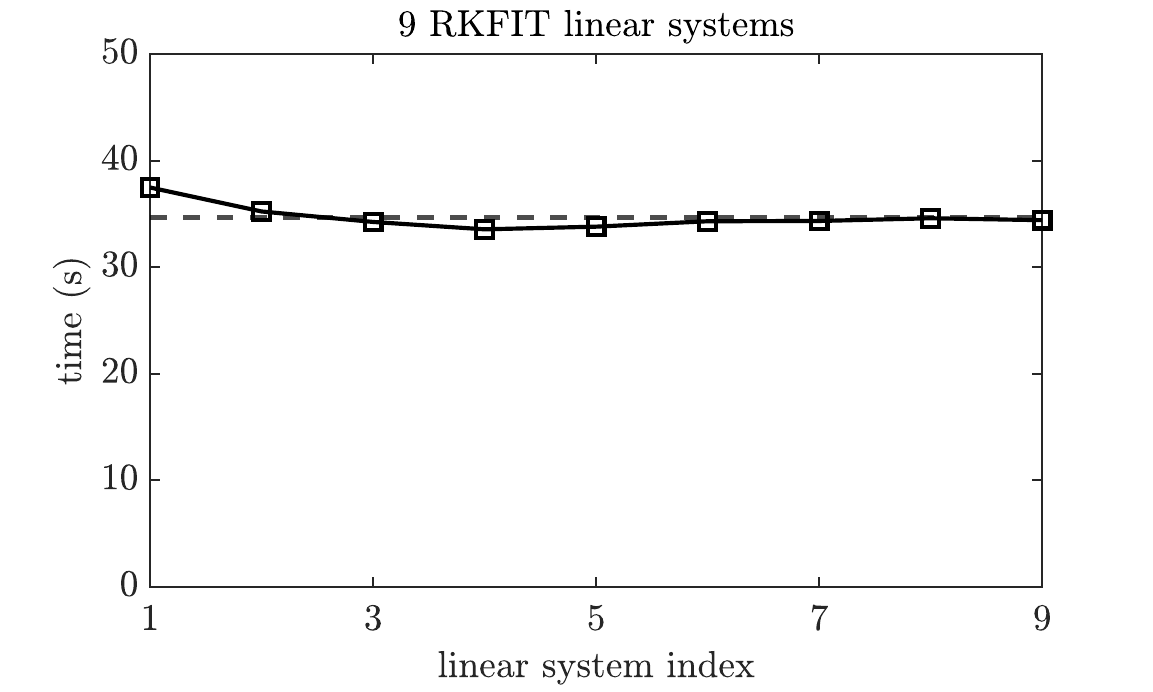}
    \end{subfigure}
    
    \caption{\textit{Time required to solve each linear system using $n=21$ real Zolotarev shifts (left) compared to 9~RKFIT complex shifts (right). The dashed lines indicate the average time per system.}}
    \label{timing}
\end{figure}

\begin{figure}
    \centering
    \begin{subfigure}[b]{0.48\textwidth}
        \hspace*{-3mm}\includegraphics[width=1.0\textwidth]{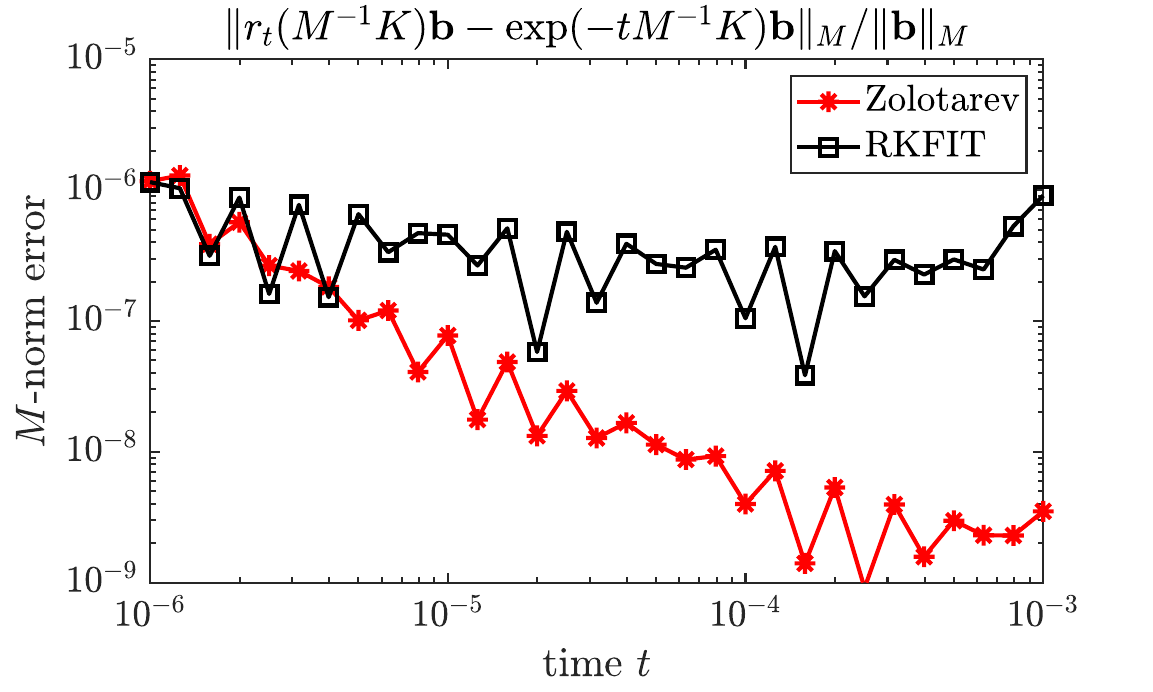}
    \end{subfigure}
    \begin{subfigure}[b]{0.48\textwidth}
        \includegraphics[width=1.0\textwidth]{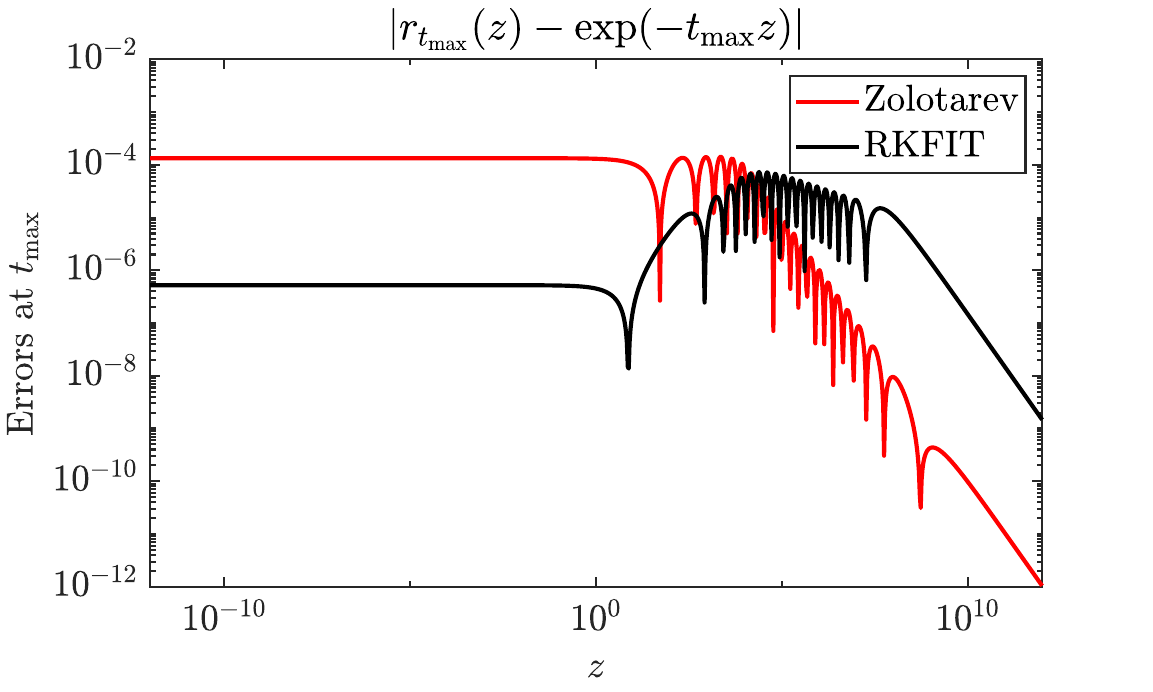}
    \end{subfigure}
    \caption{\emph{Left: The relative $M$-norm error of the Zolotarev and RKFIT approximants for the \texttt{tem181302} problem. Right: The error curves of the Zolotarev and RKFIT approximants at $t_{\max}$.} }
    \label{Mnorm_err}
\end{figure}

As a generalization of the relative $2$-norm error in~\eqref{t_err_bnd}, the left plot of Figure~\ref{Mnorm_err} shows the relative $M$-norm error 
\begin{align}
 \mathrm{rel.\ err} = \frac{\left\|r_t(M^{-1}K)\bm{b}-\exp(-tM^{-1}K)\bm{b}\right\|_M}{\left\|\bm{b}\right\|_M} \label{rel_err}
\end{align}
at each time $t$, where we compute each $\exp(-tM^{-1}K)\bm{b}$ using the Carath{\'e}o\-dory--Fej{\'e}r complex-pole rational approximant with degree $\widehat{n} = 13$. Applying~\cite[Theorem~3.2]{borner2015three} and using the fact that $M^{-1}K$ is similar to a symmetric positive semidefinite matrix, this is also bounded above by the scalar error
\begin{equation}
\mathrm{rel.\ err} \leq \max_{z\geq 0}\left|r_{t,n}(z)-\exp(-tz)\right|. \label{rel_err_bnd}
\end{equation}
In this example, we use a different time interval $T = [10^{-6}, 10^{-3}]$, corresponding to the time scale of our TEM problem. However, this time range gives the same scalar error bound~\eqref{rel_err_bnd} as the previous time interval $[10^{-3}, 1]$, since the time ratio $\tau:=t_{\max}/t_{\min}$ remains the same.

We observe in Figure~\ref{Mnorm_err} (left)  that the relative $M$-norm error with the Zolotarev poles is at least two orders of magnitude below that of using the RKFIT poles for late time points. This can be explained by the fact that the eigenvalues of $M^{-1}K$ are located mainly in the region where the scalar error of the RKFIT approximant is larger than that of the Zolotarev approximant for large values of $t$. The right plot of Figure~\ref{Mnorm_err} shows the scalar error functions of the two approximants at the final time point $t_{\max}$, from which we can see that the error of the Zolotarev approximant decays to zero, filtering out the large eigenvalues of $M^{-1}K$ more quickly.

\subsection{Total error bound}\label{teb}

The bounds~\eqref{t_err_bnd} and~\eqref{rel_err_bnd} assume that the partial fraction is evaluated exactly. In practice, however, the floating-point error must be taken into account once the error incurred in evaluating the partial fraction becomes comparable to the scalar approximation error. We use the term \emph{total error bound} to refer to the sum of the scalar approximation bound (the right-hand side of~\eqref{t_err_bnd} or~\eqref{rel_err_bnd}) and the floating-point error bound (the sum of the bounds on $E_1,E_2,E_3$ derived in   section~\ref{error_anal}).

Figure~\ref{fig:total_error} shows the total error bound for degrees $n=1,2,\ldots,60$ and time ratios~$\tau$ ranging from $10^1$ to $10^4$. In this experiment, we make one slight modification: the optimal pole interval is chosen by minimizing an augmented objective function consisting of the discrete time-uniform error function~\eqref{minimize} together with the floating-point error bound on $E_1$ from Theorem~\ref{err1-bnd}, where the residual norm of each shifted linear system is uniformly replaced by $10^{-15}$, a reasonable approximation for direct solvers with double-precision. The bounds on $E_2$ and $E_3$, given in Theorems~\ref{round_err_resid} and~\ref{err2_bnd}, are not included in the augmented objective function because they are typically negligible compared to $E_1$ due to the unit round-off factors involved. Moreover, both terms depend on the norms of the computed solution vectors, for which no uniform estimator is available. The matrix and vector used here are the same as in section~\ref{err_anal_exper}.

Figure~\ref{fig:total_error} shows that, for all time ratios, the total error bound initially decays at essentially the same rate as the scalar time-uniform error in Figure~\ref{tue_all_T}. This is because the floating-point error is negligible for small degrees. As the degree increases, however, the total error bound ceases to decrease, eventually stagnating at approximately $10^{-6}$ for all time ratios. This demonstrates that the proposed uniform rational approximation of the matrix exponential guarantees at least six digits of accuracy, regardless of the length of the time interval.

We emphasize that the total error bound is only an \emph{upper} bound and may be conservative in practice. Nevertheless, it provides a useful guide for exploiting parallelism with \emph{direct solvers}. For example, numerical experiments show that, for a time ratio of $10^3$, the norm of residues can grow as large as $O(10^9)$ for degree $n=30.$ Inspecting the bound on $E_1$ in Theorem~\ref{err1-bnd} then suggests that  single-precision accuracy (about $O(10^{-8})$) in the final result can be achieved by solving the shifted linear systems with direct methods in double-precision arithmetic.

\begin{figure}
    \centering
    \includegraphics[width=0.7\linewidth]{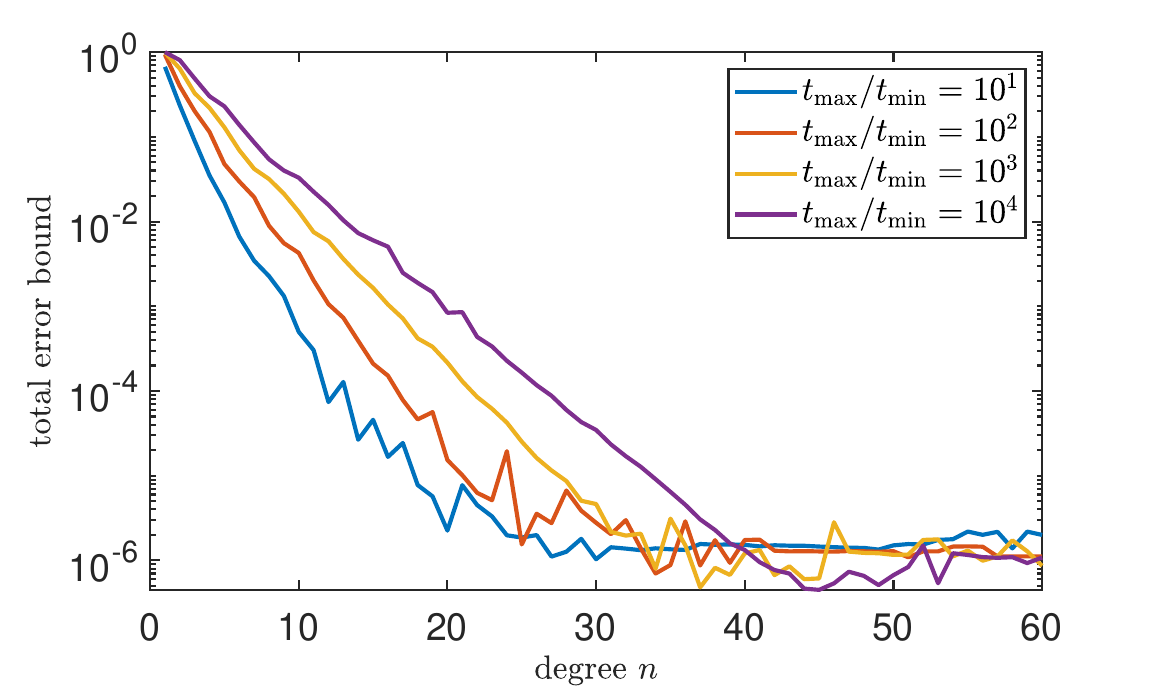}
    \caption{\emph{Total error bound for various degrees $n$ and time ratios $t_{\max}/t_{\min}$. Each data is taken from the maximum in the time interval.} }
    \label{fig:total_error}
\end{figure}

\section{Conclusions}\label{concl}

We have proposed two new approaches for the uniform-in-time approximation of the matrix exponential with shared real poles. The first approach considers the case of concentrated poles, for which we derive an explicit formula for the optimal pole when a suitable weight function is used. The second approach addresses the case of distinct poles, for which we provide a simple Zolotarev-based algorithm to efficiently select near-optimal poles.

Our floating-point error analysis shows that the partial fraction form of the Zolotarev approximants can be evaluated accurately when the shifted linear systems are solved using direct methods. The analysis demonstrates that 
time-uniform approximants can be evaluated in parallel with at least six digits of accuracy using double-precision arithmetic, over a wide range of time intervals. This analysis is based on upper bounds which may be quite pessimistic. In our experience, it is reasonable that 
\emph{with double-precision direct solvers we can expect at least single-precision accuracy in the solutions computed in parallel.} 

We believe that this is an important result as it provides retrospective justification of the common practice (e.g., in 3D geophysical simulations) of using double-precision direct solvers to compute quantities that are likely subject to much higher levels of discretization errors. We can now argue that this supposedly over-accurate solution of the linear systems actually enables a high degree of parallelization on the partial fraction level.

\section*{Acknowledgments}
S.\,G. acknowledges funding from the UK's Engineering and Physical Sciences Research Council (EPSRC grant EP/Z533786/1) and the Royal Society (RS Industry Fellowship IF/R1/231032). S.\,S. acknowledges a  Dean's Doctoral Scholarship from the University of Manchester.

\bibliographystyle{siamplain}
\bibliography{references}
\end{document}